\RequirePackage[l2tabu,orthodox]{nag}
\documentclass[a4paper,10pt]{amsart}
\usepackage{etex}
\usepackage[T1]{fontenc} 
\usepackage[utf8]{inputenx}
\usepackage{lmodern} 
\usepackage{multicol}
 
\usepackage[kerning=true]{microtype}

\author{Benjamin McKay}
\title[Hartogs and parabolic geometries]%
{The Hartogs extension problem for holomorphic parabolic and reductive geometries}
\date{\today}

\address{School of Mathematical Sciences, University College Cork, Cork, Ireland}
\email{b.mckay@ucc.ie} 

\thanks{It is a pleasure to thank Sorin Dumitrescu and Filippo Bracci for helpful conversations on the problems solved in this paper, and for inviting me to the Laboratoire de Math{\'e}matiques J.A. Dieudonn{\'e} at the University of Nice Sophia--Antipolis and to the University of Rome Tor Vergata where part of this work was carried out. This publication has emanated from activity conducted with the financial support of Science Foundation Ireland under the International Strategic Cooperation Award Grant Number SFI/13/ISCA/2844.}

\usepackage{mathtools}

\usepackage{array}
\usepackage{amsfonts}
\usepackage{euscript}
\usepackage{standalone}
\usepackage{braket}
\usepackage{comment}
\usepackage{amsmath}
\usepackage{xstring}
\usepackage{amscd}
\usepackage{amssymb}
\usepackage{euscript}
\usepackage{amsthm}
\usepackage{booktabs}
\usepackage{pdftexcmds}
\usepackage{tikz-cd}
\usepackage[draft]{varioref}
\usepackage{tikz}
\usetikzlibrary{decorations.markings}
\usepackage{graphics}
\usepackage{xspace}
\usepackage{amsxtra}
\usepackage{mathrsfs}

\newtheorem{theorem}{Theorem}[section]
\newtheorem{lemma}[theorem]{Lemma}
\newtheorem{corollary}[theorem]{Corollary}
\newtheorem{proposition}[theorem]{Proposition}
\theoremstyle{remark}
{%
    \newtheorem{example}[theorem]{Example}
}%
\newtheorem{definition}[theorem]{Definition}

\newcommand*{\defeq}{\mathrel{\vcenter{\baselineskip0.5ex \lineskiplimit0pt
                     \hbox{\scriptsize.}\hbox{\scriptsize.}}}%
                     =}
\newcommand*{\C}[1]{\ensuremath{\mathbb{C}^{#1}}}

\newcommand*{\Z}[1]{\ensuremath{\mathbb{Z}^{#1}}}
\newcommand*{\pr}[1]{\ensuremath{\left(#1\right)}}
\newcommand*{\of}[1]{\ensuremath{\pr{#1}}}
\newcommand*{\br}[1]{\ensuremath{\left\{#1\right\}}}
\newcommand*{\wo}[1]{\ensuremath{\!-#1}}
\newcommand*{\dimC}[1]{\ensuremath{\dim_{\mathbb{C}}\of{#1}}}

\newcommand*{\Sym}[2]{\ensuremath{\operatorname{Sym}^{#1}\pr{#2}}}
\newcommand*{\GL}[1]{\ensuremath{\operatorname{GL}\!\pr{#1}}}

\newcommand*{\LieSL}[1]{\ensuremath{\mathfrak{sl}\!\pr{#1}}}

\newcommand*{\PSL}[1]{\ensuremath{\mathbb{P}\operatorname{SL}\!\pr{#1}}}
\newcommand*{\nForms}[2]{\ensuremath{\Omega^{#1}\!\pr{#2}}}
\newcommand*{\Lm}[2]{\ensuremath{\Lambda^{#1}\!\pr{#2}}}
\newcommand*{\Lmtop}[1]{\ensuremath{\Lm{\text{top}}{#1}}}
\newcommand*{\cohomology}[2]{\ensuremath{H^{#1}\!\pr{#2}}}
\DeclareMathOperator{\Ad}{Ad}
\DeclareMathOperator{\ad}{ad}

\newcommand*{\hook}{\ensuremath{\mathbin{ \hbox{\vrule height1.4pt width4pt depth-1pt \vrule height4pt width0.4pt depth-1pt}}}}
\newcommand*{\Proj}[1]{\ensuremath{\mathbb{P}^{#1}}}

\makeatletter
\newcommand*{\pd}[3][1]
{
\frac{%
\partial%
\ifnum\pdf@strcmp{#1}{1}=0\else^#1\fi%
#2}%
{%
\partial
\ifnum\pdf@strcmp{#1}{1}=0\else^#1\fi%
#3
}%
}
\makeatother

\newcommand*{\OO}[1]{%
  \ensuremath{%
    \mathcal{O}%
    \IfStrEq{#1}{0}{}{\of{#1}}
  }%
}%
\newcommand*{\OOp}[2]{
  \ensuremath{
    \mathcal{O}
    \IfStrEq{#1}{0}{}{\of{#1}}
    \IfStrEq{#2}{1}{}{^{\oplus{#1}}}
  }
}

\makeatletter
\renewcommand*\env@matrix[1][\arraystretch]{%
  \edef\arraystretch{#1}%
  \hskip -\arraycolsep
  \let\@ifnextchar\new@ifnextchar
  \array{*\c@MaxMatrixCols c}}
\makeatother

\newcounter{remarkCounter}
\includecomment{longversion}
\excludecomment{shortversion}

\newcommand*{\git}{\!\!\mathbin{
  \mathchoice{/\mkern-6mu/}
    {/\mkern-6mu/}
    {/\mkern-5mu/}
    {/\mkern-5mu/}}\!\!}

\newlength{\transposeHeight}

\newcommand*{\lb}[2]{\ensuremath{\left[#1#2\right]}}

\makeatletter
\newcommand*{\map}[3][:]%
{\ensuremath{\ifnum\pdf@strcmp{#1}{:}=0\else{#1}\colon\fi{#2} \to {#3}}}
\newcommand*{\mapto}[3][:]%
{\ensuremath{\ifnum\pdf@strcmp{#1}{:}=0\else{#1}\colon\fi{#2} \mapsto {#3}}}
\newcommand*{\Hom}[2]{\ensuremath{\operatorname{Hom}\left({#1},{#2}\right)}}
\makeatother

\newcommand*{\Lie}[1]{\ensuremath{\mathfrak{\lowercase{#1}}}}
\newcommand*{\LieA}{\Lie{A}}
\newcommand*{\LieG}{\Lie{G}}
\newcommand*{\LieH}{\Lie{H}}

\newcommand*{\LieM}{\Lie{M}}
\newcommand*{\LieN}{\Lie{N}}
\newcommand*{\LieP}{\Lie{P}}

\newcommand*{\prodquot}[3]{\ensuremath{#1 \times^{#3}\! #2}}

\newcommand*{\zentrum}[1]{\ensuremath{Z_{#1}}}

\newcommand*{\fundamentalgroup}[1]{\ensuremath{\pi_1\of{#1}}}

\newcommand*{\sumnoncptneg}[1]%
{
\ensuremath%
{%
\sum_{#1 \in \Delta^-_N}%
}%
}
\newcommand*{\sumnoncptpos}[1]%
{
\ensuremath%
{%
\sum_{#1 \in \Delta^+_N}%
}%
}

\newcommand*{\bigwedgenoncptneg}[1]%
{
\ensuremath%
{%
\bigwedge_{#1 \in \Delta_N^-}
}%
}

\newcommand*{\KillingForm}[2]{\ensuremath{\left<#1,#2\right>}}

\newcommand*{\kc}
{\ensuremath{\phi}}

\newcommand*{\Fr}[1]{F#1}

\newcommand*{\fol}{\ensuremath{\mathscr{F}}}

\newcommand*{\ConnVar}{\ensuremath{\Lambda}}

\begin{document}

\begin{abstract}
Every holomorphic effective parabolic or reductive geometry on a domain over a Stein manifold is the pullback of a unique such geometry on the envelope of holomorphy of the domain. 
We use this result to classify the Hopf manifolds which admit holomorphic reductive geometries, and to classify the Hopf manifolds which admit holomorphic parabolic geometries.
Every Hopf manifold which admits a holomorphic parabolic geometry with a given model admits a flat one.
We classify flat holomorphic parabolic geometries on Hopf manifolds.
For every generalized flag manifold there is a Hopf manifold with a flat holomorphic parabolic geometry modelled on that generalized flag manifold.
\end{abstract}

\maketitle
{
\begin{multicols}{2}
\tableofcontents 
\end{multicols}
}

\section{Statement of the theorems}

For definitions of terms in this section, see section~\ref{section:Definitions}.
In this paper, we show that holomorphic parabolic geometries obey the Hartogs extension phenomenon, which constrains the possible singularities they can have, and completes the open problems of my paper \cite{McKay:2009}.
By contrast, many examples \cite{McKay:2009} of holomorphic Cartan connections do \emph{not} obey the Hartogs extension phenomenon.
Extension of parabolic and reductive geometries across codimension 2 singularities is elementary, as is extension to envelope of holomorphy for reductive geometries, but extension to envelope of holomorphy for parabolic geometries turns out to be more difficult, and requires some new results relating parabolic geometries to holomorphic connections on the canonical bundle.
We also provide new examples of holomorphic parabolic geometries on compact complex manifolds.
For infinitely many generalized flag varieties, there were previously no torsion-free parabolic geometries modelled on that generalized flag variety known on any compact complex manifold other than the model.
We use Hartogs extension to classify the Hopf manifolds which admit holomorphic reductive geometries, and to classify the Hopf manifolds which admit holomorphic parabolic geometries.
Every Hopf manifold which admits a holomorphic parabolic geometry with a given model admits a flat one.
We classify flat holomorphic parabolic geometries on Hopf manifolds.
For every generalized flag variety, we construct a Hopf manifold with a flat holomorphic parabolic geometry modelled on that generalized flag variety.

\begin{theorem}\label{theorem:Reductive}
Every holomorphic effective reductive geometry with connected model on a domain over a Stein manifold is the pullback of a unique such geometry from the envelope of holomorphy of the domain.
\end{theorem}

\begin{theorem}\label{theorem:Parabolic}
Every holomorphic effective parabolic geometry on a domain over a Stein manifold is the pullback of a unique such geometry from the envelope of holomorphy 
of the domain.
\end{theorem}

\subsection{Application: blowing down}\label{section:blow.up}

\begin{proposition}\label{proposition:blow.down}
A complex manifold bearing a holomorphic effective parabolic geometry cannot blowdown.
More generally, suppose that \(f \colon X \to Y\) is a holomorphic map of complex manifolds, and that away from some subvariety \(X_0 \subset X\), \(f\) is a biholomorphism, say \(\left.f\right|_{X-X_0} \colon X-X_0 \to Y-Y_0\), where \(Y_0 \subset Y\) is a subvariety of complex codimension 2 or more.
Suppose further that there is a point \(x \in X\) near which \(f\) is not a local biholomorphism.
For example, you might imagine that \(X\) is a blowup of \(Y\) along \(Y_0\).
Then \(X\) bears no holomorphic effective parabolic geometry.
\end{proposition}
\begin{proof}
Suppose for the moment that \(X\) is equipped with a holomorphic effective parabolic or reductive geometry.
Then \(f\) induces such a geometry on \(Y-Y_0\).
By our theorems above, this geometry holomorphically extends to \(Y\).
The bundle and Cartan connection of the geometry on \(Y\) pull back to \(X\), agreeing by analyticity with the geometry already found on \(X\).
The pullback of the Cartan connection is therefore nondegenerate along \(X_0\).
But then \(f\) is a local biholomorphism everywhere on \(X\).
\end{proof} 
This is the only known obstruction to existence of a parabolic geometry on any noncompact complex manifold.
\begin{corollary}
Suppose that \(M\) is a compact complex manifold of algebraic dimension zero with Albanese \(A_M\).
Suppose that \(\dim A_M=\dim M\).
Note that dimension, algebraic dimension and Albanese dimension are birational invariants of \(M\). 
Suppose that \(M\) has a holomorphic parabolic geometry.
Then \(M\) is a complex torus and the parabolic geometry is translation invariant.
\end{corollary}
\begin{proof}
The Albanese map \(M \to A_M\) is a proper modification \cite{Ueno:1975} p. 159 corollary 13.7.
By proposition~\vref{proposition:blow.down}, the Albanese map is an isomorphism.
Every holomorphic parabolic geometry on any complex torus is translation invariant \cite{McKay:2008}.
\end{proof}

\subsection{All known parabolic geometries on compact complex manifolds}

Suppose that \pr{X,G} is a rational homogeneous variety.
Let us describe all known examples of holomorphic \pr{X,G}-geometries on compact complex manifolds.
There is a unique holomorphic \pr{X,G}-geometry on \(X\) \cite{McKay:2004}.
All flat holomorphic \pr{X,G}-geometries are classified on compact complex surfaces \cite{Klingler:1998}.
All holomorphic parabolic geometries on complex tori are classified \cite{McKay:2008}; every holomorphic parabolic geometry on any complex torus is translation invariant.
An \pr{X,G}-\emph{Klein geometry} is a compact complex manifold \(M\) arising as quotient \(M=\Gamma \backslash U\) of an open subset \(U \subset X\) by a discrete group \(\Gamma \subset G\) acting freely and properly; it bears the obvious \pr{X,G}-structure (see section~\vref{section:structures}).
Every 3-dimensional projective Klein geometry \(M=\Gamma\backslash U\) for which \(U\subset \Proj{3}\) contains a projective line is known \cite{Kato:2010}.
On any connected complex manifold containing a rational curve, any holomorphic projective connection is flat \cite{Biswas/McKay:2010} p. 9 corollary 4, has a Zariski dense family of deformations of the rational curve, and admits no holomorphic deformations, as the Schwarzian derivative has to vanish along the rational curves \cite{McKay:2013c}, so each of these Klein geometries is the unique holomorphic projective connection on each of those complex 3-folds.

Recall that a \emph{cominiscule variety} is a Hermitian symmetric space \(X\) together with a connected complex semisimple Lie group \(G\) of biholomorphisms of \(X\), so that \(G\) contains the isometry group of \(X\) as a maximal compact subgroup.
There is a real form \(G^* \subset G\) of \(G\) which acts on \(X\) with orbit a noncompact symmetric space: the noncompact dual \(X^* \subset X\) of \(X\).
There are many examples of Hermitian locally symmetric manifolds \cite{Mok:1989}, i.e. compact complex manifolds with \(X^*\) as universal covering space, or equivalently \pr{X^*,G^*}-Klein geometries.
Each Hermitian locally symmetric manifold has a flat holomorphic \pr{X,G}-geometry given by the inclusion \(X^* \subset X\).
This \((X,G)\)-geometry is the unique holomorphic cominiscule geometry on that complex manifold \cite{McKay:2013c}.
Every compact complex manifold \(M\) with \(c_1(M,T)<0\) which admits a holomorphic cominiscule geometry is a locally Hermitian symmetric space \cite{McKay:2013c}.
There is one known exotic example \cite{Jahnke/Radloff:2004} of a cominiscule geometry, exotic because it is defined on a smooth complex projective variety, but is not (i) \(X\) or (ii) a complex torus or (iii) a locally Hermitian symmetric manifold; in this one example, \(\pr{X,G}=\pr{\Proj{3},\PSL{4,\C{}}}\).
In each of our examples so far, all of the known geometries are flat, except that there are both flat and nonflat holomorphic \pr{X,G}-geometries on any complex torus of the same dimension as \(X\) \cite{McKay:2008}.

Next suppose that \pr{X,G} is a rational homogeneous variety, but is \emph{not} cominiscule.
Every complex torus of the same dimension as \(X\) admits \pr{X,G}-geometries; although all of these geometries are translation invariant, none of them are flat \cite{McKay:2008}.

Suppose that \(P' \subset P \subset G\) are parabolic subgroups of a complex semisimple Lie group, and that \(X=G/P\) is cominiscule. 
Let \(X'\defeq G/P'\); consider the fiber bundle map \(P/P' \to X' \to X\).
Take a complex manifold \(M\) with holomorphic cominiscule modelled on \(\pr{X,G}\),  say \(E \to M\).
Let \(M'\defeq E/P'\), giving a holomorphic fiber bundle \(P/P' \to M' \to M\) whose total space \(M'\) comes with a flat holomorphic \(\pr{X',G}\)-geometry, with bundle \(E \to M'\) and with the same Cartan connection as \(E \to M\).
This geometry is the only holomorphic parabolic geometry known on any such complex manifold \(M'\).
The only known examples of holomorphic noncominiscule parabolic geometries are (i) the model \(X\), (ii) complex tori and (iii) these bundles \(M' \to M\).
Those on complex tori cannot be flat; torsion-free holomorphic noncominiscule parabolic geometries on tori are not known to exist.

\section{Definitions}\label{section:Definitions}

\subsection{Envelopes of holomorphy}

A complex manifold \(M\) is \emph{Stein} if \(M\) admits a proper holomorphic embedding in a finite dimensional complex Euclidean space. 
A connected complex manifold \(M\) is a \emph{domain over} a Stein manifold \(X\) if there is a local biholomorphism \(M \to X\).
Take the space \(C\) of characters of the algebra of holomorphic functions on \(M\), with the weak\({}^*\) topology, and map \(M \to C\) by associating to a point the evaluation character of that point.
Extend each holomorphic function on \(M\) to a function on \(C\) by evaluating characters.
The \emph{envelope of holomorphy} \(\hat{M} \subset C\) of a domain \(M\) over a Stein manifold is the connected component of \(C\) containing the image of \(M\).
Rossi \cite{Rossi:1963} proved that if \(M\) is a domain over a Stein manifold then \(\hat{M}\) has a unique structure of a Stein manifold so that the inclusion \(M \to \hat{M}\) is an open holomorphic embedding, and every holomorphic function on \(M\) is pulled back from a unique holomorphic function on \(\hat{M}\).
For an introduction to the theory of extension and envelopes of holomorphy see \cite{Matsushima/Morimoto:1960}.

\subsection{Cartan geometries}

If \(E \to M\) is a principal right \(H\)-bundle, we denote the action of any element \(h \in H\) as \(r_h \colon E \to E\).

Suppose that \(G\) is a Lie group acting transitively on a manifold \(X\) and let \(H \subset G\) be the stabilizer of a point \(o \in X\).
Denote the Lie algebras of \(H \subset G\) by \(\LieH \subset \LieG\). 
An \(\pr{X,G}\)-geometry, also known as a \emph{Cartan geometry} modelled on \pr{X,G}, on a manifold \(M\) is a choice of \(C^\infty\) principal \(G\)-bundle \(E_G \to M\) with connection \(\omega\) and \(H\)-subbundle \(E \subset E_G\) so that the tangent spaces of \(E\) are complementary to the null spaces of \(\omega\).
The \emph{Cartan connection} is the restriction of \(\omega\) to \(E\), also denoted \(\omega\).
Each \(A \in \LieG\) has associated vertical vector field on \(E_G\), which splits at each point of \(E\) by complementarity of \(E\) and the null spaces of \(\omega\): denote the projection to \(T_e E\) as \(\vec{A}\).
The bundle \(E \to M\) and the Cartan connection \(\omega\) determine the Cartan geometry as \(E_G \defeq \prodquot{E}{G}{H}\) and \(\omega\) extends uniquely to \(E_G\) to become a connection.

Sharpe \cite{Sharpe:2002} gives an introduction to Cartan geometries.
\begin{example}
The principal \(H\)-bundle \(G \to X\) is an \((X,G)\)-geometry, with the left invariant Maurer--Cartan 1-form \(\omega=g^{-1} \, dg\) on \(G\) as Cartan connection; this geometry is called the \emph{model \((X,G)\)-geometry}.
\end{example}
An \emph{isomorphism} of \(\pr{X,G}\)-geometries \(E_0 \to M_0\) and \(E_1 \to M_1\) with Cartan connections \(\omega_0\) and \(\omega_1\) is an \(H\)-equivariant diffeomorphism  \(\map[F]{E_0}{E_1}\) so that \(F^* \omega_1 = \omega_0\).

\subsection{Effective, reductive and parabolic geometries}

A complex linear algebraic group \(G\) is \emph{reductive} if it contains a Zariski dense compact subgroup.
A \emph{generalized flag variety}, also called a \emph{rational homogeneous variety}, is a
complex projective variety \(X\) acted on transitively and holomorphically by a connected complex semisimple Lie group \(G\).
The  stabilizer \(H \subset G\) of a point \(o \in X\) is a closed complex subgroup called a \emph{parabolic subgroup}.
An \(\pr{X,G}\)-geometry is \emph{effective} if \(G\) acts effectively on \(X\), \emph{parabolic} if \(G\) is semisimple and \(H\) is parabolic, and \emph{reductive} if \(G\) is a linear algebraic group and \(H\) is a reductive linear algebraic subgroup.

\subsection{The tangent bundle in a Cartan geometry}

If \(X\) and \(Y\) are manifolds acted on by a Lie group \(G\), let \(\prodquot{X}{Y}{G}\) be the quotient of \(X \times Y\) by the diagonal \(G\)-action.

\begin{lemma}[Sharpe \cite{Sharpe:1997} p. 188, theorem 3.15]\label{lemma:TgtBundle}
If \(\pi \colon E \to M\) is any Cartan geometry,
say with model \(X=G/H\), then the diagram
\[
\begin{tikzcd}
0 \arrow{r} & \ker \pi'(e) \arrow{r} \arrow{d}{\omega_e} & T_e E \arrow{r} \arrow{d}{\omega_e} & T_m M \arrow{r} \arrow{d} & 0 \\
0 \arrow{r} & \LieH \arrow{r} & \LieG \arrow{r} &
\LieG/\LieH \arrow{r} & 0
\end{tikzcd}
\]
commutes for any points \(m \in M\) and \(e \in E_m\); thus
\[
TM=\prodquot{E}{\pr{\LieG/\LieH}}{H}
\text{ and }
T^*M=\prodquot{E}{\pr{\LieG/\LieH}^*}{H}.
\]
Under this identification, vector fields on \(M\) are identified with \(H\)-equivariant functions \(E \to \LieG/\LieH\), and sections of the cotangent bundle with \(H\)-equivariant functions \(E \to \pr{\LieG/\LieH}^*\).
\end{lemma}

\subsection{Structures}\label{section:structures}

An \emph{\pr{X,G}-pair} on a manifold \(M\) is a \emph{developing map} \(\delta \colon \tilde{M} \to X\), a local diffeomorphism from the universal covering space \(\tilde{M} \to M\), equivariant under a \emph{holonomy morphism} \(h \colon \fundamentalgroup{M} \to G\), a group morphism.
Two \pr{X,G}-pairs \pr{\delta,h} and \pr{\delta',h'} are \emph{equivalent} if \(\pr{\delta',h'}=\pr{g\delta,ghg^{-1}}\) for some \(g \in G\).
An \emph{\pr{X,G}-structure} is an equivalence class of \pr{X,G}-pairs.
An \pr{X,G}-structure induces a flat \pr{X,G}-geometry with bundle \(E = \delta^* G/\fundamentalgroup{M}\) and connection form pulled back from the model.
If \(\pr{X,G}\) is effective then every flat \pr{X,G}-geometry arises from a unique \pr{X,G}-structure.
The reader should be aware of possible terminological confusion between \(G\)-structures and \pr{X,G}-structures.

\section{Extension phenomena}

\begin{longversion}

\begin{theorem}[Matsushima, Morimoto \cite{Matsushima/Morimoto:1960} p. 146 theorem 2]
A complex Lie group with finitely many path components is Stein if and only if the identity component of the center of the identity component is isomorphic to a product \(\pr{\C{\times}}^p \times \C{q}\) for some integers \(p, q \ge 0\).
\end{theorem}

\begin{theorem}[Matsushima, Morimoto \cite{Matsushima/Morimoto:1960} p. 147 theorem 4]\label{theorem:mat.mor.up}
If a holomorphic principal bundle has Stein structure group and Stein base, then it has Stein total space.
\end{theorem}

\begin{theorem}[Matsushima, Morimoto \cite{Matsushima/Morimoto:1960} p. 151 theorem 5, Fischer \cite{Fischer:1969} p. 341]\label{theorem:mat.mor.down}
If a holomorphic principal bundle has reductive structure group and Stein total space, then it has Stein base space.
\end{theorem}
We note that by definition a reductive complex Lie group has finitely many path components.
Matsushima and Morimoto prove their result with the additional hypothesis of connected structure group, but their proof works identically without that hypothesis.

\end{longversion}

\begin{lemma}\label{lemma:bundle.envelope}
Suppose that \(M\) is a domain over a Stein manifold, with envelope of holomorphy \(f \colon M \to \hat{M}\).
Take a reductive Lie group \(G\) and a holomorphic principal bundle \(G \to E \to \hat{M}\).
Then \(E\) is a Stein manifold and \(f^*E\) is a domain over \(E\) with the obvious map \(f^*E \to E\) as envelope of holomorphy.
\end{lemma}
\begin{proof}
\begin{longversion}
By theorem~\vref{theorem:mat.mor.up}, \(E\) is Stein.
\end{longversion}
\begin{shortversion}
By \cite{Matsushima/Morimoto:1960} p. 147 theorem 4, \(E\) is Stein.
\end{shortversion}
The map \(f^*E \to E\) displays \(f^*E\) as a domain over a Stein manifold.
Let \(f^*E \to E'\)  be the envelope of holomorphy.
The map \(f^*E \to E\) factors uniquely through \(E'\), because the holomorphic functions on \(E\) separate points and pullback to \(f^*E\) and then extend to \(E'\).
By uniqueness of the envelope of holomorphy, \(f^*E \to E' \to E\) is \(G\)-equivariant.
Therefore \(G\) acts properly holomorphically and freely on \(E'\), say with quotient \(M'\), a complex manifold factoring \(M \to M' \to \hat{M}\).
\begin{longversion}
By theorem~\vref{theorem:mat.mor.down}, \(M'\) is Stein.
\end{longversion}
\begin{shortversion}
By \cite{Matsushima/Morimoto:1960} p. 151 theorem 5, \(M'\) is Stein.
\end{shortversion}
By uniqueness of the envelope of holomorphy and the factorization \(M \to M' \to \hat{M}\),  \(M' \to \hat{M}\) is a biholomorphism, so \(E'=E\).
\end{proof}

\begin{proposition}%
\label{proposition:vector.bundle.sections}%
Suppose that \(M\) is a domain over a Stein manifold, with envelope of holomorphy \(f \colon M \to \hat{M}\) and that \(V \to \hat{M}\) is a holomorphic vector bundle.
Then every holomorphic section of \(f^* V \to M\) is the pullback of a unique holomorphic section of \(V \to \hat{M}\).
\end{proposition}
\begin{proof}
Let \(\hat{E} \to \hat{M}\) be the associated principal bundle (with structure group \(G=\GL{N,\C{}}\) where \(N\) is the rank of \(V \to \hat{M}\)), and let \(E \defeq f^* \hat{E}\). 
Every holomorphic section of \(f^*V \to M\) is uniquely identified with a \(G\)-equivariant holomorphic map \(E \to \C{N}\).
Each of these maps is pulled back from a unique map on \(\hat{E}\) by lemma~\vref{lemma:bundle.envelope}, so that unique map is \(G\)-equivariant, and thus represents a section of \(V\).
\end{proof}

\begin{longversion}
\begin{lemma}\label{lemma:hypersurfaceIntersections}
Suppose that \(M\) is a domain in a Stein manifold, with envelope of holomorphy \(\hat{M}\). Then \(M \to \hat{M}\) intersects every closed complex hypersurface in \(\hat{M}\).
\end{lemma}
\begin{proof}
Let \(H \subset \hat{M}\) be a closed complex hypersurface.
Henri Cartan's Theorem A (H\"ormander \cite{Hormander:1990} p. 190 theorem 7.2.8) says that for every coherent sheaf \(F\) on any Stein manifold every fiber \(F_m\) is generated by 
global sections. 
Let \(\OO{H}\) be the line bundle corresponding to the hypersurface \(H\).
By Cartan's theorem, for each point \(m \in H\), there is a global section of \(\OO{H}\) not vanishing at \(m\). 
Take such a section, and let \(H'\) be its zero locus.

Henri Cartan's theorem B (H\"ormander \cite{Hormander:1990} p. 199 theorem 7.4.3) says that the cohomology of any coherent sheaf on any Stein manifold vanishes in positive degrees.
By the exponential sheaf sequence, (see H\"ormander \cite{Hormander:1990} p. 201)
\[\cohomology{1}{\hat{M},{\mathcal{O}}^{\times}}=\cohomology{2}{\hat{M},\Z{}}.\]
In other words, line bundles on \(\hat{M}\) are precisely determined by their cohomology classes. 
Therefore a divisor is the divisor of a meromorphic function just when it has trivial
first Chern class. 
In particular, \(H'-H\) is the divisor of a meromorphic function, say \(h \colon \hat{M} \to \C{}\).
So \(h\) is holomorphic on \(\hat{M} \setminus H\), and has poles precisely on \(H\).
If \(H\) does not intersect the image of \(M\) in \(\hat{M}\), then the pullback of \(h\) to \(M\) is holomorphic on \(M\), so is the pullback of a holomorphic function on \(\hat{M}\), i.e. has no poles, so \(H\) is empty.
\end{proof}
\end{longversion}

\begin{lemma}\label{lemma:AtiyahClassZero}
Suppose that \(M\) is a domain over a Stein manifold.
Then every holomorphic principal bundle (and every holomorphic vector bundle) on \(M\) which is pulled back from \(\hat{M}\) admits a holomorphic connection.
\end{lemma}
\begin{proof}
A holomorphic principal (or vector) bundle admits a holomorphic connection if and only if its Atiyah class vanishes \cite{Atiyah:1957}. 
The Atiyah class of each holomorphic principal (or vector) bundle lives in the first cohomology group of a coherent sheaf.
The first cohomology group of any coherent sheaf on any Stein manifold vanishes by 
Henri Cartan's theorem B \cite{Hormander:1990}.
Therefore every holomorphic principal bundle on \(\hat{M}\) admits a holomorphic connection.
Pullback such a connection to the pullback bundle over \(M\).
\end{proof}

\section{First order structures}

Suppose that \(M\) is a complex manifold and that \(V\) is a complex vector space, and that 
\(\dimC{M} =\dimC{V}\). 
The \emph{\(V\)-valued coframe bundle} \(FM\) is the set of all pairs \((m,u)\) where \(m \in M\) and \map[u]{T_m M}{V} is a complex linear isomorphism;  \(FM \to M\) is a principal right \(\GL{V}\)-bundle, under the action \(r_g (m,u) = \pr{m,g^{-1} u}\) for \(g \in \GL{V}\) and \((m,u) \in FM\). 
We suppress mention of \(V\); the reader can then assume that \(V=\C{n}\).
Let \map[p]{FM}{M} be the bundle map \mapto{(m,u)}{m}. 
The \emph{soldering form} of \(FM\) is the 1-form \(\eta \in \nForms{1}{FM} \otimes V\) given by
\(\eta_{(m,u)} = u \circ \pi'(m,u)\).

\subsection{\texorpdfstring{$G$-structures}{G-structures}}

Suppose that \(G\) is a complex Lie group and that \(V\) is a complex \(G\)-module. 
Suppose that \(M\) is a complex manifold, \(\dim M=\dim V\), and \(FM\) is the \(V\)-valued coframe bundle of \(M\).
A \emph{\(G\)-structure} on \(M\) is a holomorphic principal right \(G\)-bundle \(B \to M\) with a \(G\)-equivariant holomorphic bundle map \(B \to FM\).
If we wish to disregard the group \(G\), we will refer to a \(G\)-structure as a \emph{first order structure}.

\begin{longversion}
\subsection{Extension}

\begin{lemma}[Hilbert \cite{Procesi:2007} p. 556]\label{lemma:Hilbert}
Suppose that \(H \subset G\) is a closed reductive linear algebraic subgroup of a complex linear algebraic group. 
There is a finite dimensional \(G\)-module \(W\) and a nonzero vector \(w_0 \in W\) with stabilizer \(H\).
Consequently, \(G/H \subset W\) is an affine variety. 
Specifically, we can take \(W\) to be the dual space of the set of \(H\)-invariant regular functions on \(G\) of degree at most \(k\), for any sufficiently large positive integer \(k\).
\end{lemma}
\begin{proof}
The categorical quotient \(G\git{H}\) (spectrum of \(H\)-invariant polynomials) of an affine algebraic variety by a reductive algebraic group is an affine algebraic variety, embedded in \(W\) for sufficient large degree \(k\); Procesi \cite{Procesi:2007} p. 556, theorem 2. 
The categorical quotient will in general only parameterize the closed orbits, and admit an equivariant holomorphic submersion \(G/H \to G\git{H}\).
But the \(H\)-orbits in \(G\) are the translates of \(H\), so all orbits are closed, and thus \(G/H=G\git{H}\).
\end{proof}

\end{longversion}

\begin{longversion}
If \(G \subset \GL{V}\) is a closed subgroup, then every \(G\)-structure \map{B}{FM}
is a submanifold of \(FM\).
Quotienting the map \map{B}{FM} by right \(G\)-action yields a holomorphic map \map{M}{FM/G}, a holomorphic section of the holomorphic fiber bundle \map{FM/G}{M}. 
Conversely, given any holomorphic section \map[s]{FM/G}{M}, note that \(FM \to FM/G\) is also a holomorphic principal right \(G\)-bundle, and let \(B \defeq s^* FM\). 
Then \map{B}{FM} is a holomorphic \(G\)-structure. 
If \(G\) is a closed subgroup of \(\GL{V}\), then holomorphic \(G\)-structures are identified with holomorphic sections of \(FM/G \to M\).

Suppose that \map[f]{M_0}{M_1} is a local biholomorphism, that \map[\pi]{B_1}{M_1} is a holomorphic principal right \(G\)-bundle, and that \map[h_1]{B_1}{FM_1} is a holomorphic \(G\)-structure.
Let \(B_0 \defeq f^* B_1\) be the \emph{pullback}, i.e. \(B_0\) is the set of pairs \(\pr{m_0,b_1}\) so that \(\pi\of{b_1}=f\of{m_0}\).
Define a map \map[h_0]{B_0}{FM_0} by \(h_0\of{m_0,b_1}=h_1(b_1) \circ f'\of{m_0}\).
If \(G \subset \GL{V}\) is a closed subgroup, then \(B_1 \subset FM_1\) is a submanifold, and 
\(B_1 = f_1 B_0\), where 
\[
f_1\of{m_0,u_0}=\pr{f\of{m_0},u_0 \circ f'\of{m_0}^{-1}}.
\]

\end{longversion}

\begin{theorem}\label{theorem:reductive.G.structure}
Suppose that \(V\) is a finite dimensional complex vector space and \(G \subset \GL{V}\) is a closed reductive complex algebraic subgroup.
Then every holomorphic \(G\)-structure on a domain \(M\) over a Stein manifold with \(\dimC{M}=\dimC{V}\) is pulled back from a unique holomorphic \(G\)-structure on the envelope of holomorphy \(\hat{M}\) of \(M\).
\end{theorem}
\begin{proof}
\begin{longversion}
By lemma~\vref{lemma:Hilbert}, 
\end{longversion}
\begin{shortversion}
Clearly 
\end{shortversion}
\(\GL{V}\!/G \subset W\) is an affine variety and an orbit of \(\GL{V}\) in a \(\GL{V}\)-module \(W\).
Form the associated vector bundle \(W_M \defeq \prodquot{FM}{W}{\GL{V}}\), containing the fiber subbundle \(E_M \defeq \prodquot{FM}{\pr{\GL{V}\!/G}}{\GL{V}}= FM/G\).
Every \(G\)-structure is a holomorphic section of \(E_M \subset W_M\), so is pulled back from a global section of \(W_{\hat{M}} = \prodquot{F\hat{M}}{W}{\GL{V}}\).
By analytic continuation, the image lies inside the fiber subbundle \(E_{\hat{M}} = \prodquot{F\hat{M}}{\pr{\GL{V}\!/G}}{\GL{V}}\).
Reverse the argument: this section represents a \(G\)-structure on \(\hat{M}\) pulling back to the original one on \(M\).
\end{proof}

\section{Extending by extending bundles}

The \emph{soldering form} of a \(G\)-structure \(B \to FM\) is the pullback of the soldering form on \(FM\).

\begin{longversion}
\begin{lemma}
Suppose that \(E \to M\) is a holomorphic principal bundle over a domain \(M\) over a Stein manifold.
Then \(E \to M\) admits at most one extension to a holomorphic principal bundle over  \(\hat{M}\), up to canonical isomorphism.
\end{lemma}
\begin{proof}
The identity map \(E \to E\) extends to a morphism, as does its inverse.
\end{proof}
\end{longversion}

\begin{theorem}\label{theorem:ExtendByFirstOrderStructure}
Suppose that \(G\) is a complex Lie group and that \(V\) is a finite dimensional complex analytic \(G\)-module.
Suppose that \(M\) is a domain over a Stein manifold with a holomorphic \(G\)-structure \(E \subset FM\).
This \(G\)-structure is pulled back from a unique \(G\)-structure on the envelope of holomorphy \(\hat{M}\) if and only if the holomorphic principal bundle \(E \to M\) is pulled back from a holomorphic principal bundle on \(\hat{M}\).
\end{theorem}
\begin{proof}
Suppose that \(\hat{E} \to \hat{M}\) is a holomorphic principal bundle pulling back to \(E \to M\).
The soldering form \(\eta\) is a holomorphic section of the holomorphic vector bundle \(T^*E \otimes^G \C{n}\) over \(M\), and therefore pulled back from a holomorphic section of \(T^*\hat{E} \otimes^G \C{n} \to \hat{M}\) (also denoted \(\eta\)) by proposition~\vref{proposition:vector.bundle.sections}.
This \(\eta\) vanishes on the fibers of \(E \to M\) so by analytic continuation vanishes on the fibers of \(\hat{E} \to \hat{M}\). 
The set of points \(m \in \hat{M}\) above which \(\eta\) is not of full rank as a linear map \[T_e\hat{E}/T_e\pr{\hat{E}_m} \to \C{n}\] is a closed complex analytic hypersurface.
Every component of every closed complex analytic hypersurface in \(\hat{M}\) intersects the image of \(M\) 
\begin{longversion}
by lemma~\vref{lemma:hypersurfaceIntersections}.
\end{longversion}
\begin{shortversion}
\cite{McKay:2009} lemma~11, p. 8.
\end{shortversion}
But above \(M\), there is no point where \(\eta\) has less than full rank.
Therefore this hypersurface is empty. 
The soldering form, being semibasic, determines a coframe at each point \(e \in \hat{E}\), say \(u \colon T_m M \to \C{n}\) so that \(v \hook \eta_e = u\of{\pi'(e)v}\) for any vector \(v \in T_e E\).
This choice of coframe is an extension of \(E \to FM\) to a map \(\hat{E} \to F\hat{M}\). 
The map \(\hat{E} \to F\hat{M}\) is \(G\)-equivariant, so a \(G\)-structure.
\end{proof}

\begin{lemma}\label{lemma:Cartan.pulled}
Suppose that \(M\) is a domain over a Stein manifold, with envelope of holomorphy \(f \colon M \to \hat{M}\) and holomorphic Cartan geometry \(E \to M\).
The Cartan geometry \(E \subset FM\) is pulled back from a holomorphic Cartan geometry on the envelope of holomorphy \(\hat{M}\) (uniquely up to a unique isomorphism) if and only if the holomorphic principal bundle \(E \to M\) is pulled back from a holomorphic principal bundle on \(\hat{M}\).
\end{lemma}
\begin{proof}
Suppose that the Cartan geometry has model \(\pr{X,G}\) where \(X=G/H\).
Let \(\hat{V}\defeq \prodquot{T^* \hat{E}}{\LieG}{H}\) and let \(V\defeq \prodquot{T^* E}{\LieG}{H}\).
The Cartan connection of \(E\) is the pullback of a unique holomorphic section of \(\hat{V}\) by proposition~\vref{proposition:vector.bundle.sections}, hence to an \(H\)-equivariant \(\LieG\)-valued holomorphic 1-form \(\omega\) on \(\hat{E}\) which pulls back to the Cartan connection on \(E\).
The identity \(\vec{A} \hook \omega=A\) is satisfied on \(E\), so by analytic continuation is satisfied on \(\vec{E}\).
Consider the determinant \(\det \omega\), a holomorphic section of \(\prodquot{(\det T^*\hat{E})}{\det \LieG}{H}\) and let \(Z\) be the zero locus of \(\det \omega\) in \(\hat{M}\).
By definition of a Cartan connection, \(\det \omega \ne 0\) on \(M\), i.e. the hypersurface \(Z\) does not strike the image of \(M\) in \(\hat{M}\).
\begin{longversion}
By lemma~\vref{lemma:hypersurfaceIntersections}, 
\end{longversion}
\begin{shortversion}
By \cite{McKay:2009} lemma~3.24, p. 8,
\end{shortversion}
\(Z\) is empty.
Therefore \(\omega\) on \(\hat{E}\) is a Cartan connection.
\end{proof}

For example, suppose that \(\pr{X,G}\) is a complex homogeneous space and that \(X\) is a domain over a Stein manifold but is not Stein.
As a particular case, we can consider \(X=\C{n}\wo{0}\), \(G=\GL{n,\C{}}\), \(n\ge 2\), so \(\hat{X}=\C{n}\).
Then \(G \to X\) is a holomorphic principal bundle, but does not admit a holomorphic extension to a holomorphic principal bundle on \(\hat{X}\).

\section{Cartan geometries and their induced first order structures}

Suppose that \(G\) is a complex Lie group and that \(H \subset G\) is a closed complex subgroup and \(X=G/H\).
Suppose that \(\map[\pi]{E}{M}\) is a holomorphic \pr{X,G}-geometry with Cartan connection \(\omega\). 
To each point \(e \in E\), we can associate an isomorphism \(\map[\omega_e]{T_e E}{\LieG}\) and an isomorphism \(\map[\omega_e  + \LieH]{T_m M}{\LieG/\LieH}\), where \(m=\pi(e)\)%
\begin{longversion}
, by lemma~\vref{lemma:TgtBundle}.
\end{longversion}
\begin{shortversion}
.
\end{shortversion}
Define \(\map[f]{E}{FM}\) by \(f(e) = \omega_e+\LieH\).
Let \(H_1 \subset H\) be the subgroup of elements of \(H\) which act trivially on \(\LieG/\LieH\). Clearly \(f\) descends to an \(\pr{H/H_1}\)-structure \map[f]{E/H_1}{FM}, the \emph{associated first order structure} of the Cartan geometry.
If \(H/H_1 \subset \GL{\LieG/\LieH}\) is a closed subgroup then the associated first order structure has image an \(H/H_1\)-reduction of \(FM\).

\begin{lemma}\label{lemma:Faith}
Suppose that \((X,G)\) is an effective complex homogenous space and that \(X=G/H\) is connected.
Then every reductive subgroup \(H' \subset H\) acts effectively on \(\LieG/\LieH\).
\end{lemma}
\begin{proof}
Because \(H'\) is reductive, \(\LieG\) splits into a sum of \(H'\)-modules, say \(\LieG = \LieH \oplus \LieH^{\perp}\) with \(\LieH^{\perp} = \LieG/\LieH\) as \(H\)-modules, so \(H'_1\) acts trivially on \(\LieH^{\perp}\).
Therefore the subgroup \(H'_1 \subset H'\) commutes with the subgroup \(G_0 \subset G\) generated by exponentiating \(\LieH^{\perp}\).
Again since \(\LieH^{\perp} = \LieG/\LieH\), \(G_0\) acts on \(X\) with an open orbit, say \(X_0=G_0/H_0 \subset X\), where \(H_0 = G_0 \cap H\). 
Clearly \(H'_1\) fixes every element of \(X_0\). 
Since \(X\) is connected, and \(G\) acts analytically on \(X\), \(H'_1\) fixes all elements of \(X\). 
But \(G\) acts effectively, so \(H'_1=\left\{1\right\}\).
\end{proof}

\begin{corollary}\label{corollary:EequalsEOne}
Suppose that \(G\) is a complex Lie group and \(H \subset G\) is a closed reductive subgroup and that \(X=G/H\) is connected and that \(G\) acts effectively on \(X\).
Suppose that \(E \to M\) is an \pr{X,G}-geometry. 
Then the first order structure induced by that geometry is an embedding \(E \to FM\) as a principal right \(H\)-subbundle.
\end{corollary}

\section{Induced Cartan geometries}

A \emph{morphism} of homogeneous spaces \((X,G) \to \pr{X',G'}\) is a Lie group morphism \(G \to G'\) and an equivariant smooth map \(X \to X'\).
If \(X \to X'\) is a local diffeomorphism, then every \pr{X,G}-geometry \(E \to M\) induces an \pr{X',G'}-geometry \(E' \defeq E \times^H H' \to M\), where \(H \subset G\) and \(H' \subset G'\) are stabilizers of points \(x_0 \in X \mapsto x_0' \in X'\), and with \(\omega'\) being the composition of \(\omega\) with the Lie algebra morphism \(\LieG \to \LieG'\) on the image of the tangent spaces of \(E\), and equal to the Maurer--Cartan 1-form on each fiber of \(E' \to M\).
It is easy to check that there is a unique such 1-form \(\omega'\), and that it is a Cartan connection.

\section{Reductive reductions and the Sharpe geometry}\label{section:Sharpe.geometry}

Suppose that \(G\) is a Lie group with Lie algebra \(\LieG\), \(H \subset G\) a closed subgroup with Lie algebra \(\LieH\) and let \(X\defeq G/H\). 
Suppose that \(H' \subset H\) is a reductive subgroup.
Split \(\LieG\) as an \(H'\)-module, say \(\LieG = \pr{\LieG/\LieH} \oplus \LieH' \oplus \pr{\LieH/\LieH'}\).
Write this splitting as \(A = A_- + A_0 + A_+\) for \(A \in \LieG\).
Let \(\LieG'\defeq \pr{\LieG/\LieH} \oplus \LieH'\) be the unique Lie algebra which has the obvious \(H'\)-module structure, but has \(\left[{\LieG/\LieH},{\LieG/\LieH}\right]=0\).
The \emph{Sharpe space} \pr{X',G'} of \pr{X,G} has \(G'=\pr{\LieG/\LieH} \ltimes H'\), with obvious Lie group structure, and \(X'=G'/H'=\LieG/\LieH\).
If we don't specify \(H'\), take \(H' \subset H\) to be any maximal reductive subgroup (i.e. Levi subgroup).
As an example in the category of real smooth manifolds, the Sharpe space of hyperbolic space is Euclidean space.
For a complex analytic example, the Sharpe space of a smooth affine quadric \(X=\Set{a \in \LieSL{2,\C{}}|\det a = 1}\) under its group of regular algebraic automorphisms \(G = \PSL{2,\C{}}\) is \(X'=\C{3}\), \(G'=\C{\times} \times \C{2}\), affine space acted on by rescaling and translations.

For any \pr{X,G}-geometry, say \(E \to M\), with Cartan connection \(\omega\) and with an \(H'\)-reduction \(E' \subset E\), the \emph{associated Sharpe geometry} is defined as follows: define the bundle of the geometry to be \(E' \to M\) and the Cartan connection of the Sharpe geometry to be \(\omega'=\omega_- + \omega_0\) via the splitting of the Lie algebra \(\LieG\) as above.
Note that \(\omega_0\) is a connection on \(E'\).

On the other hand, given any \pr{X',G'}-geometry, say \(E' \to M\), with Cartan connection \(\omega'\), define a bundle \(E \defeq E' \times^{H'} H\) and a 1-form \(\omega\) to be the unique \(H\)-equivariant 1-form on \(E\) valued in \(\LieG\) which equals \(\omega'\) on \(E'\), i.e. the induced \pr{X,G}-geometry, and \(E' \to M\) is the Sharpe geometry of the reduction \(E' \subset E \to M\).

\begin{lemma}\label{lemma:reductive.geometry.induces.affine.connection}
Suppose that \((X,G)\) is a complex homogeneous space  and \(H' \subset G\) is a reductive subgroup fixing a point of \(X\).
If a complex manifold \(M\) admits a holomorphic \pr{X,G}-geometry \(E \to M\) and a holomorphic \(H'\)-reduction \(E' \subset E\) then \(E \to M\) admits a holomorphic connection and \(M\) admits a holomorphic affine connection.
\end{lemma}
\begin{proof}
The holomorphic connection \(\omega_0\) on \(E'\) induces a holomorphic connection on \(E=\prodquot{E'}{H}{H'}\) and the induced first order structure \(E' \to FM\) induces a holomorphic affine connection on \(M\).
\end{proof}

\section{Extension of reductive geometries}

\begin{theorem}\label{theorem:ExtendFaithful}
Suppose that \((X,G)\) is a complex homogeneous space and \(H \subset G\) is the \(G\)-stabilizer of a point of \(X\), with Lie algebras \(\LieH \subset \LieG\).
Suppose that \(H\) acts faithfully on \(\LieG/\LieH\). 
Suppose that \(M\) is a domain over a Stein manifold, and \(E \to M\) is a holomorphic \pr{X,G}-geometry.
Then the \pr{X,G}-geometry is pulled back from the envelope of holomorphy \(\hat{M}\) if and only if the induced first order structure is pulled back from \(\hat{M}\).
\end{theorem}
\begin{proof}
The Cartan geometry on \(M\) is pulled back from \(\hat{M}\) just when the associated principal \(H\)-bundle \(E \to M\) is pulled back, by lemma~\vref{lemma:Cartan.pulled}.
The bundle of the associated first order structure is \(E/H_1=E\) by corollary~\vref{corollary:EequalsEOne}.
By theorem~\vref{theorem:ExtendByFirstOrderStructure}, the holomorphic principal bundle \(E \to M\) is the pullback of a holomorphic principal bundle on \(\hat{M}\) if and only if the first order structure extends holomorphically to \(\hat{M}\).
\end{proof}

\begin{lemma}\label{lemma:extend.by.reducing}
Suppose that \pr{X,G} is an effective complex homogeneous space with \(X=G/H\) connected and that \(H' \subset H\) is a reductive subgroup.
Suppose that \(E \to M\) is a holomorphic \pr{X,G}-geometry on a domain \(M\) over a Stein manifold.
If \(E \to M\) admits a holomorphic \(H'\)-reduction then the \pr{X,G}-geometry on \(M\) is pulled back from the envelope of holomorphy \(\hat{M}\).
\end{lemma}
\begin{proof}
The Sharpe geometry on the \(H'\)-reduction \(E' \subset E\) is effective and reductive.
By~\vref{theorem:ExtendFaithful}, the Sharpe geometry is pulled back from \(\hat{M}\) just when its underlying reductive first order structure is pulled back from \(\hat{M}\).
By~\vref{theorem:reductive.G.structure}, the underlying first order structure is pulled back from \(\hat{M}\).
Therefore \(E'\) extends to a holomorphic principal \(H'\)-bundle \(E' \to \hat{M}\), and so \(E=\prodquot{E'}{H}{H'}\) extends to a holomorphic principal \(H\)-bundle \(E \to \hat{M}\).
By lemma~\vref{lemma:Cartan.pulled} the \pr{X,G}-geometry extends.
\end{proof}

Theorem~\vref{theorem:Reductive} follows clearly.

\section{Rational homogeneous varieties}

\begin{shortversion}
We take notation and terminology from \cite{Knapp:2002}.
\end{shortversion}
\begin{longversion}
Suppose that \((X,G)\) is a rational homogeneous variety.
The Lie algebra \(\LieP\) of \(P\) contains a unique Cartan subalgebra, say with root system \(\Delta\), with a basis of simple roots so that the root vectors of the positive simple roots lie in \(\LieP\).
Let \(\Delta^+\) be the set of positive roots.
Associate to \(P\) the sets: 
\begin{itemize}
\item[] \(\Delta_P\): the set of roots whose root vectors lie in \(\LieP\),
\item[] \(\Delta_c\): the set of compact roots, i.e. \(\Delta_P \cap -\Delta_P\),
\item[] \(\Delta_N^+\): the set of noncompact positive roots,
\item[] \(\Delta_N^-\): the set of noncompact negative roots.
\end{itemize}
The Dynkin diagram of \(P\), or of \((X,G)\), is that of \(G\) with the compact simple roots marked as dots and the noncompact simple roots marked as crosses.

\begin{lemma}%
\label{lemma:Langlands}
Suppose that \(P \subset G\) is a parabolic subgroup of a connected complex semisimple Lie group. 
Then \(P\) is a connected complex linear algebraic subgroup. 
There are unique connected closed complex linear algebraic subgroups \(M, A, N \subset P\) and \(N^- \subset G\) so that 
\begin{enumerate}
\item \(M\) is a maximal connected semisimple subgroup of \(P\) and
\item \(A\) is an abelian subgroup of \(P\) and
\item \(N\) is the unipotent radical of \(P\) and
\item \(P=MAN\), the \emph{Langlands decomposition}, i.e. every element \(p \in P\) is expressible uniquely as a product \(p = man\) with \(ma \in MA\), \(n \in N\).
Moreover, this decomposition is a biholomorphism \(P = MA \times N\) and 
\item \(1 \to M \cap A \to M \times A \to MA \to 1\) is an finite extension by an abelian group \(M \cap A\) and
\item \(N^-\cap P = \left\{1\right\}\) and 
\item the Lie algebra of \(M\) is the span of the compact coroots and the root spaces of the compact roots and
\item the Lie algebra of \(A\) is the span of the noncompact coroots and
\item the Lie algebra of \(N\) is the sum of the noncompact positive root spaces and
\item the Lie algebra of \(N^-\) is the sum of the noncompact negative root spaces and
\item the compact roots form a root system for \(M\) and
\item the groups \(N\) and \(N^-\) are simply connected.
\end{enumerate}
\end{lemma}
The group \(M\) is the complex semisimple Lie group with Dynkin diagram given by deleting the crossed nodes from the Dynkin diagram of \(P\).
\begin{proof}
Parabolic subgroups are connected \cite{Borel:1991} p. 154 theorem 11.16.
Let \(N\) be the unipotent radical of \(P\) and \(H \subset P\) the Cartan subgroup.
Take \(A\) to be the identity component of \(P \cap H\) and let \(L\) be the centralizer of \(A\) in \(G\); \(L\) is a reductive complex linear algebraic subgroup \(L \subset P\) so that \(L \cap N=\left\{1\right\}\), \(A \subset L\) and \(P = L \ltimes N\), an isomorphism of complex linear algebraic groups and so also of complex algebraic varieties \cite{Humphreys:1975} p. 185.
Since \(P\) is connected, \(N\) and \(L\) are also connected.
The Lie algebra of \(L\) is the sum of the Cartan subalgebra and the root spaces of the compact roots \cite{Humphreys:1975} p. 185.

Let \(M=\lb{L}{L}\) be the derived group and let \(\zentrum{L}\) be the center of \(L\).
Because \(L\) is connected, \(M\) is also connected. 
It happens that \(M\) is a closed complex linear algebraic subgroup of \(L\) \cite{Borel:1991} p. 58 section 2.3 unnumbered corollary.
Because \(L\) is reductive, \(L\) is generated by its maximal semisimple \(M\) and its center \(\zentrum{L}\), and \(M \cap \zentrum{L}\) is finite \cite{Humphreys:1975} p. 168.
The map \(\pr{m,a} \in M \times A \mapsto ma \in L\) is a holomorphic group morphism, and an isomorphism of Lie algebras, and \(L\) is connected, so this map is surjective with finite abelian kernel \(M \cap A\).
\end{proof}

\subsection{Roots and characters}\label{subsection:Roots.and.characters}

Given a character \(\chi \in \Hom{P}{\C{\times}}\), \(\chi=1\) on the derived group so \(\chi=0\) on \(\LieM \oplus \LieN\).
Thus the  characters of \(P\) are the characters of \(A\) that are trivial on \(M \cap A\).
Every 1-dimensional \(P\)-module is identified up to isomorphism with a character, having \(P\)-action \(pv=\chi(p)v\) for a unique character \(\chi \in \Hom{P}{\C{}}\).
The characters of \(P\) correspond bijectively to the \(G\)-equivariant holomorphic line bundles on \(X\).
Let \(\alpha_i\) denote the simple roots and let \(\omega_i\) denote the fundamental weights:
\[
\delta_{ij} = 2\frac{\KillingForm{\alpha_i}{\omega_j}}{\KillingForm{\alpha_i}{\alpha_i}}.
\]
Every character \(\chi\) can be written uniquely as
\[
\chi = \sum_i p_i \omega_i,
\]
for some integers \(p_i\), where \(i\) runs over the noncompact simple roots \(\alpha_i\), so \(\chi\of{H_i}=p_i\) on the coroot \(H_i=\check{\alpha}_i\).
On the other hand, any sum \(\sum_i p_i \omega_i\) has some multiple trivial on the finite abelian group \(M\cap A\), so a character.

If \(V\) is a \(P\)-module, the \emph{character} \(\chi_V\) of \(V\) is the character of \(\Lmtop{V}\). 
For example,
\[
\chi=\chi_{\LieN}
=
\sum_{\beta \in \Delta^+_N} \beta.
\]

Inside the Weyl chamber we find the cone of \emph{dominant \(P\)-characters} \(\chi\): those with \(p_i \ge 0\) for all \(i\).
\end{longversion}
\begin{shortversion}
Take a dominant \(P\)-character \(\chi \in \Hom{P}{\C{\times}}\).
\end{shortversion}
If \(\chi=0\) on some coroot \(H\) of a noncompact root \(\alpha\), we extend \(\chi\) to be \(0\) on the root spaces of \(\pm \alpha\), so that \(\chi\) is be defined on a parabolic subgroup \(P_{\chi} \subset G\) with \(P \subset P_{\chi}\) and let \(X_{\chi}\defeq G/P_{\chi}\).
\begin{longversion}
Draw each dominant \(P\)-character \(\chi\) as the Dynkin diagram of \(P_{\chi}\) but with an eigenvalue \(p_i \in \Z{}\) over each noncompact simple root \(\alpha_i\).
The \(G\)-equivariant morphism \(X \to X_{\chi}\) pulls back the associated \(G\)-equivariant line bundle.
Draw \(\chi\) as a vector: a sum of fundamental weights in the weight lattice of \(\LieG\).
The compact roots of \(P_{\chi}\) are those perpendicular to \(\chi\).
The dominant \(P\)-characters span a cone.
A dominant \(P\)-character \(\chi\) lies in the interior of that cone just when \(X_{\chi}=X\).
For \(\chi\) along the walls of the cone, \(X_{\chi}\) goes down in dimension on each wall \(p_i=0\), down further still in dimension on the corners, i.e. more dots in the Dynkin diagram.
Each wall of the dominant \(P\)-character cone corresponds to a cross in the Dynkin diagram of \((X,G)\) becoming a dot in the diagram of \(\pr{X_{\chi},G}\).
The intersection of two walls corresponds to two crosses becoming dots, etc. 
The Langlands decomposition \(P_{\chi}=M_{\chi} A_{\chi} N_{\chi}\) satisfies \(M \subset M_{\chi}\), \(A_{\chi} \subset A\), \(N_{\chi} \subset N\).

\subsection{Chevalley basis}

\begin{definition}\label{definition:ChevalleyBasis}
Suppose that \(G\) is a complex semisimple Lie group and \(H \subset G\) is a Cartan subgroup, with associated Lie algebras \(\LieH \subset \LieG\).
A \emph{Chevalley basis} \cite{Chevalley:1955,Serre:2001} is a basis \(X_{\alpha}, H_{\alpha_i}\) for \(\LieG\) parameterized by roots \(\alpha\) and simple roots \(\alpha_i\) for which
\begin{enumerate}
\item \(\lb{H}{X_{\alpha}}=\alpha(H) X_{\alpha}\) for each \(H \in \LieH\)
\item \(\alpha\of{H_{\beta}}=2 \frac{\KillingForm{\alpha}{\beta}}{\KillingForm{\beta}{\beta}}\) (measuring inner products via the Killing form)
\item \(\lb{H_{\alpha}}{H_{\beta}}=0\),
\item
\[
\lb{X_{\alpha}}{X_{\beta}}=
\begin{cases}
H_{\alpha}, & \text{ if } \alpha+\beta=0, \\
N_{\alpha\beta} X_{\alpha+\beta}, & \text{otherwise} \\
\end{cases}
\]
with
\begin{enumerate}
\item \(N_{\alpha\beta}\) an integer,
\item \(N_{\beta,\alpha}=-N_{\alpha,\beta}\),
\item \(N_{-\alpha,-\beta}=-N_{\alpha\beta}\),
\item If \(\alpha, \beta, \) and \(\alpha+\beta\) are roots, then
\(N_{\alpha\beta}=\pm (p+1)\),
where \(p\) is the largest integer for which \(\beta-p \, \alpha\) is a root,
\item
\(N_{\alpha \beta}=0\) if \(\alpha+\beta = 0\) or if any of \(\alpha, \beta,\) or  \(\alpha+\beta\) is not a root.
\end{enumerate}
\end{enumerate}
\end{definition}

Every Cartan subgroup \(H\) of every complex semisimple Lie group \(G\) has a Chevalley basis \cite{Carter:1972} p. 56 theorem 4.2.1.

\begin{lemma}\label{lemma:NNMinusGenerate}
Suppose that \(G/P\) is a generalized flag variety, with Langlands decomposition \(P=MAN\).
Then \(G\) is generated by \(N^{-} \cup N\) if and only if the maximal connected normal subgroup of \(G\) lying in \(P\) is trivial.
\end{lemma}
\begin{proof}
Let \(G' \subset G\) be the subgroup generated by \(N^{-} \cup N\). 
Since \(N^{-}\) and \(N\) are path connected, \(G'\) is connected, so a Lie subgroup. 
Let \(\LieG''\) be the Lie algebra generated by \(\LieN^- \cup \LieN\).
Since \(\LieN^-\) and \(\LieN\) are complex Lie subalgebras of \(\LieG\), so is \(\LieG''\).
Let \(G''\) be the connected Lie subgroup with Lie algebra \(\LieG''\).
Because \(G''\) has complex Lie algebra, \(G''\) is a complex Lie subgroup of \(G\).
Since \(N^-\) and \(N\) are connected, \(N^-, N \subset G''\).
Clearly the Lie algebra \(\LieG'\) of \(G'\) contains \(\LieN^-\) and \(\LieN\), so contains \(\LieG''\).
Therefore \(G'' \subset G'\).
But \(N^-, N \subset G'\) so \(G' \subset G''\).
Therefore \(G'\) is a complex Lie subgroup of \(G\).

In terms of a Chevalley basis, the Lie algebra \(\LieG'\) of \(G'\) contains all root vectors \(X_{\alpha}\) of all noncompact roots \(\alpha\), and therefore their brackets. 
The noncompact roots, positive or negative, are invariant under bracketing with the 
compact roots, so \(N\) and \(N^{-}\) are each invariant under \(MA\). 
Therefore \(G'\) is invariant under \(MA\). 
Therefore \(\LieG'\) is invariant under the Cartan subgroup of \(G\), and so is a sum of root spaces. 
We only need to identify which root spaces lie in \(\LieG'\). 
So far, we have found that \(\LieG'\) contains all of the root spaces of the noncompact roots.

Suppose that \(\alpha\) is a compact root. 
We need to ask if \(X_{\alpha} \in \LieG'\). 
If we can find a noncompact root \(\beta\) so that \(\alpha+\beta\) is a root, then 
\[
\lb{X_{-\beta}}{X_{\alpha+\beta}}
=
N_{-\beta,\alpha+\beta} X_{\alpha}.
\]
From the definition of Chevalley basis, \(N_{-\beta,\alpha+\beta} \ne 0\) since \(-\beta, \alpha+\beta, \alpha\) are all roots. 
Therefore, for any given compact root \(\alpha\), if there is a noncompact root \(\beta\)
so that \(\alpha+\beta\) is also a root, then \(\alpha+\beta\) is a noncompact root, so \(X_{\alpha} \in \LieG'\). 
Note that if \(\alpha\) is a compact root, and \(\beta\) a noncompact root, then all roots on the \(\alpha\)-string 
\(
\dots, \beta- \alpha, \beta, \beta+ \alpha, \dots
\)
are noncompact too, containing the same (all positive or all negative) integer multiples
of the noncompact simple roots.
In other words, if \(\alpha\) is a compact root and \(X_{\alpha}\) is not in \(\LieG'\), then the \(\alpha\)-string of every noncompact root \(\beta\) has length 1.

Consider the root system of rank 2 generated by \(\alpha\) and \(\beta\). 
Looking at pictures of all root systems of rank 2 \cite{Serre:2001}, we see that \(\alpha\) is perpendicular to \(\beta\).
So the set of compact roots \(\alpha\) for which \(X_{\alpha} \notin \LieG'\) lies on a perpendicular subspace to all of the noncompact roots, and so forms a root subsystem of a semisimple ideal in \(\LieG\).
So \(\LieG = \LieG' \oplus \LieG''\) where \(\LieG'' \subset \LieG\) is a semisimple ideal. Moreover, \(\LieG'\) contains the root spaces of all noncompact roots, so \(\LieG''\) is contained in the sum of the root spaces of the compact roots, i.e. \(G''=e^{\LieG'} \subset MA \subset P\).
By hypothesis, \(\LieG''=0\). 
Therefore \(\LieG'=\LieG\), and so \(G'\) is a connected open subgroup of \(\LieG\). 
Since \(G\) is connected, \(G'=G\).
\end{proof}

\begin{lemma}\label{lemma:MAonN}
Suppose that \(P \subset G\) is a parabolic subgroup of a connected complex semisimple Lie group and with Lie algebras \(\LieP \subset \LieG\) and with Langlands decomposition \(P=MAN\). 
The group \(MA\) normalizes the groups \(N\) and \(N^{-}\). 
In particular \(MA\) acts on \(\mathfrak{n}^{-}\) by restriction of the adjoint representation of \(G\).
Suppose that \(\LieP\) does not contain a nontrivial ideal of \(\LieG\).
Then the subgroup \(\Gamma\) defined by
\(
1 \to \Gamma \to MA \to \GL{\mathfrak{n}^{-}}
\)
is \(\Gamma=\zentrum{G}\cap MA\). In particular, \(\Gamma\) is a subgroup of the center \(\zentrum{G}\) of \(G\) and so is a finite abelian group.
\end{lemma}
\begin{proof}
If \(ma \in MA\) acts trivially on \(\mathfrak{n}^{-}\), it acts trivially on \(\mathfrak{n}\), since \(\mathfrak{n}=\pr{\mathfrak{n}^{-}}^*\) via the Killing form, so \(ma\) acts trivially on the Lie subalgebra of \(\mathfrak{g}\) generated by \(\mathfrak{n} \cup \mathfrak{n}^{-}\). 
By lemma~\vref{lemma:NNMinusGenerate}, this subalgebra is all of \(\mathfrak{g}\).
Therefore \(\Ad(ma)=I\), so \(ma \in \zentrum{G}\).  
\end{proof}

\begin{lemma}\label{lemma:EffectivityCriteria}
A generalized flag variety \(X=G/P\) is effective if and only if it satisfies both of the conditions:
\begin{enumerate} 
\item \(G\) has trivial center and
\item the Lie algebra of \(P\) contains no nontrivial ideal in the Lie algebra of \(G\).
\end{enumerate}
\end{lemma}
Note that \(G\) has trivial center if and only if \(G\) is in adjoint form, i.e. \(G=\Ad G\).
\begin{proof}
Suppose that \(G/P\) is effective.
The Lie algebra of \(P\) contains a nontrivial ideal in the Lie algebra of \(G\) just when \(P\) contains a positive dimensional subgroup normal in \(G\), which then acts trivially on \(G/P\), so is trivial.
Clearly \(\zentrum{G}\) is discrete, abelian and normal in \(G\). 
We need to see that \(\zentrum{G} \subset P\).
A subgroup \(P\) would remain parabolic if we replaced it by the subgroup \(P\, \zentrum{G}\), and would have the same Lie algebra. 
But every parabolic subgroup is connected, and determined by its Lie algebra \cite{Fulton/Harris:1991}.
So \(\zentrum{G} \subset P\). 
But \(G\) is effective, so \(\zentrum{G}=\left\{1\right\}\).

Now suppose that \(G\) has trivial center and that the Lie algebra of \(P\) contains no
nontrivial ideal in the Lie algebra of \(G\).
Suppose that \(K \subset P\) is a normal subgroup of \(G\). 
So the Lie algebra of \(K\) is an ideal in the Lie algebra of \(G\), and therefore is trivial. 
So \(K\) is a discrete normal subgroup of \(G\). 
If \(g(t)\) is a path in \(G\) with \(g(0)=1\), and \(k \in K\), then \(g(t)kg(t)^{-1} \in K\) is constant.
Since \(G\) is connected, \(gkg^{-1}=k\) for \(g \in G\) and \(k \in K\).
So \(K \subset \zentrum{G}=\left\{1\right\}\). 
\end{proof}

By lemma~\vref{lemma:Faith}, the action of \(MA\) on \(\LieN^-\) is faithul if and only if \(G/P\) is effective.

\end{longversion}

\subsection{Some homogeneous affine spaces}

\begin{lemma}\label{lemma:radical.action}
Take a rational homogeneous variety \pr{X,G}, say \(X=G/P\) with Langlands decomposition \(P=MAN\), and a dominant \(P\)-character \(\chi\) and its associated parabolic subgroup \(P_{\chi}\) with Langlands decomposition \(P_{\chi}=M_{\chi} A_{\chi} N_{\chi}\)%
\begin{longversion}
from section~\vref{subsection:Roots.and.characters}.
\end{longversion}
\begin{shortversion}
.
\end{shortversion}
Extend \(\chi\) to vanish on \(\LieN^-\) and let
\[
\ConnVar \defeq \ConnVar_{\chi} \defeq \Set{\lambda \in \LieG^*|\left.\lambda\right|_{\LieP_{\chi}}=\chi}.
\]
Then \(N_{\chi}\) acts simply transitively on \(\ConnVar\) and the \(P\)-stabilizer of the point \(\chi \in \ConnVar\) is \(P^{\chi}=MA\pr{M_{\chi}\cap N}\).
\end{lemma}
\begin{proof}
The Lie algebra action is \(A \in \LieP, \lambda \in \ConnVar \mapsto -\lambda \circ \Ad{A}\).
Write each element \(A \in \LieP\) as
\[
A = \sum_i A^i H_i + \sum_{\alpha \in \Delta_P} A^{\alpha} X_{\alpha}.
\]
The stabilizer \(P^{\lambda} \subset P\)  has Lie algebra \(\LieP^{\lambda}\) consisting of elements \(A \in \LieP\) for which \(\lambda \circ \Ad{A}=0\).
The Lie algebra \(\LieP^{\lambda}\) is the set of solutions of the linear equations
\[
A^{-\beta} 
=
\frac{\left<\beta,\beta\right>}{2 \left<\chi,\beta\right>}
\pr{
2 A^i \frac{\left<\beta,\alpha_i\right>}{\left<\alpha_i,\alpha_i\right>}
\lambda\pr{X_{\beta}}
+
\sum_{\alpha\ne -\beta}^{\alpha \in \Delta_P}
N_{\alpha\beta} A^{\alpha} \lambda\pr{X_{\alpha+\beta}}
}
\]
for \(\beta \in \Delta_{N_{\chi}}^-\).
Solve inductively in the number of simple roots that sum up to \(\beta\), so that we see that the dimension of the stabilizer \(P^{\lambda}\) is constant as we vary \(\lambda \in \Lambda\).
In particular, \(P^{\chi}\) has Lie algebra \(\LieP^{\chi} = \LieM \oplus \LieA \oplus \pr{\LieN \cap \LieM_{\chi}}\). 
By connectivity, \(MA\pr{N\cap M_{\chi}} \subset P^{\chi}\).

The \(P\)-orbit of \(\chi\) has dimension that of \(\LieP/\LieP^{\chi}=\LieN_{\chi}\), which is the dimension of \(\ConnVar\), so an open orbit.
Since all orbits have equal dimension, they are all open, and since \(\ConnVar\) is connected there is one \(P\)-orbit.
The \(P\)-orbit of \(\chi\) is the \(N_{\chi}\)-orbit of \(\chi\), since \(MA\pr{M_{\chi} \cap N}\) fixes \(\chi\).

Next we prove that the \(N_{\chi}\)-stabilizer of any point is trivial.
The set \(\Gamma \subset N_{\chi}\) of elements of \(N_{\chi}\) fixing a point \(\lambda \in \ConnVar\) is a discrete subgroup. 
The map \(n \in N_{\chi} \mapsto n\lambda \in \ConnVar\) is a \(\Gamma\)-bundle.
The affine space \(\ConnVar\) is connected and simply connected, so \(\Gamma=\left\{1\right\}\), i.e. \(\Lambda \cong N_{\chi}\) as \(N_{\chi}\)-spaces and \(N_{\chi} \cap P^{\chi}=1\).

Suppose that \(p \in P^{\chi} \subset P=MAN\), so split \(p=man\).
We want to prove that \(p \in MA\pr{N \cap M_{\chi}}\).
Without loss of generality, \(p=n \in N\) and \(n\chi=\chi\) so \(n\) preserves \(P_{\chi}\) and so \(n \in P_{\chi}\) 
and so \(n \in \pr{M_{\chi} N_{\chi}}\).
We want to prove that \(n \in M_{\chi}\).
Decompose \(n=n_{\chi} m_{\chi}\) with \(m_{\chi} \in M_{\chi}\) and \(n_{\chi} \in N_{\chi}\).
But  \(m_{\chi}\chi=\chi\) so \(\chi=n\chi=n_{\chi} \chi\) so \(n_{\chi}=1\).
\end{proof}

\subsection{The hemicanonical character}

Suppose that \(X=G/P\) is a generalized flag variety.
\begin{longversion}
Fix a Cartan subalgebra \(\LieH \subset \LieG\) and a Chevalley basis for \(\LieG\) as in definition~\vref{definition:ChevalleyBasis}.
\end{longversion}
Use the Killing form to extend \(\alpha\) from \(\LieH\) to \(\LieG\), by splitting \(\LieG = \LieH + \LieH^{\perp}\), and taking \(\alpha=0\) on \(\LieH^{\perp}\). 
Let \(\delta = \delta_{G/P}\) be 
\[
\delta \defeq \frac{1}{2} \sumnoncptpos{\alpha} \alpha
\]
where the sum is over all noncompact negative roots. 
Clearly \(2 \delta\) is the character of \(\LieN\), since we sum over the weights of the negative noncompact roots.

\begin{lemma}%
\label{lemma:DeltaBeta}%
[Knapp \cite{Knapp:2002} p. 330, %
corollary 5.100, \cite{McKay:2008} lemma 2]
For any rational homogeneous variety, the Killing form inner product
\(\KillingForm{\delta}{\beta}=0\) just when \(\beta\) is a compact root and 
\(\KillingForm{\delta}{\beta}>0\) just when \(\beta\) is a noncompact root.
So \(\delta\) is a dominant \(P\)-character with \(X_{\delta}=X\).
\end{lemma}

\section{Holomorphic connections on line bundles}

If \(P \to E \to M\) is a holomorphic principal bundle, let \(E_1\) be the set of pairs \pr{e,\phi} so that \(e \in E\) and \(\phi \in T^*_e E \otimes \LieP\) satisfies \(\vec{A} \hook \phi=A\) for all \(A \in \LieP\).
The \emph{connection bundle} of \(E\) is \(\nabla E = E_1/P\); holomorphic sections of \(\nabla E \to M\) are holomorphic connections on \(E \to M\).
Similarly if \(V \to M\) is a holomorphic vector bundle, the connection bundle \(\nabla V \to M\) is the connection bundle of the associated principal bundle.
Pick a complex homogeneous space \(\pr{X,G}\), say \(X=G/P\), and a holomorphic \pr{X,G}-geometry \(E \to M\) with Cartan connection \(\omega\).
Take a complex analytic Lie group morphism \(\rho \colon P \to \bar{P}\) to a complex Lie group \(\bar{P}\).
Let 
\[
\ConnVar=\ConnVar_{\rho}\defeq \Set{a \in \LieG^* \otimes \bar\LieP|\left.a\right|_{\LieP}=\rho}.
\]
Any connection \(\phi\) on  \(\bar{E} = \prodquot{E}{\bar{P}}{P}\) pulls back to a 1-form on \(E\), which we can write as \(\phi=a \circ \omega\), for a unique \(H\)-equivariant map \(a \colon E \to \ConnVar\); identify connections on \(\bar{E}\) with maps \(a\), i.e. \(\nabla \bar{E} = \prodquot{E}{\ConnVar}{P}\).

\begin{proposition}\label{proposition:canonical.connection}
Suppose that \pr{X,G} is an effective rational homogeneous variety, say \(X=G/P\), with Langlands decomposition \(P=MAN\). 
Take a holomorphic \(\pr{X,G}\)-geometry \(E \to M\) on a complex manifold \(M\) and to each dominant \(P\)-character \(\chi \in \Hom{P}{\C{\times}}\) associate the holomorphic line bundle \(L_{\chi}=E \times^{\chi} \C{}\).
Let \(P^{\chi}\defeq MA\pr{M_{\chi}\cap N} \subset P\).
There is a natural isomorphism of bundles of affine spaces \(E/P^{\chi} = \nabla L_{\chi}\).
The bundle \(E\) admits a \(P^{\chi}\)-reduction if and only if \(L_{\chi}\) admits a holomorphic connection, i.e. if and only if the Atiyah class of \(L_{\chi}\) vanishes.
\end{proposition}
\begin{proof}
As above, we identify \(\nabla L_{\chi} = \prodquot{E}{\ConnVar_{\chi}}{P}\) where
\[
\ConnVar_{\chi}=\Set{a \in \LieG^*|\left.a\right|_{\LieP}=\chi}.
\]
By lemma~\vref{lemma:radical.action}, \(P\) acts transitively on \(\ConnVar\) with \(P\)-stabilizer of \(\chi\) equal to \(P^{\chi}\), i.e.\(\prodquot{E}{\ConnVar}{P}=E/P^{\chi}\).
Hence a \(P^{\chi}\)-reduction of \(E\) is precisely a section of \(E/P^{\chi}=\nabla L_{\chi}\), i.e. precisely a holomorphic section of \(\nabla L_{\chi}\).
\end{proof}

For example, a holomorphic projective connection on a complex manifold arises from a holomorphic affine connection just when the canonical bundle admits a holomorphic connection.

\newcommand{\trivChars}{\ensuremath{\Sigma}}

\begin{lemma}
Suppose that \(\pr{X,G}\) is a rational homogeneous variety and \(E \to M\) is a holomorphic \((X,G)\)-geometry on a complex manifold \(M\).
The following are equivalent:
\begin{enumerate}
\item
The canonical bundle of \(M\) admits a holomorphic connection.
\item
The tangent bundle of \(M\) admits a holomorphic connection.
\item
The bundle \(E \to M\) admits a holomorphic connection
\item
The \((X,G)\)-geometry \(E \to M\) is induced by a reductive holomorphic Cartan geometry.
\item
One of the line bundles \(L_{\chi} \to M\), for some \(P\)-character \(\chi\) positive on all noncompact coroots, admits a holomorphic connection.
\item
All of the line bundles \(L_{\chi} \to M\) for all \(P\)-characters \(\chi\) admit holomorphic connections.
\item
The dominant \(P\)-characters \(\chi\) for which \(L_{\chi} \to M\) admits a holomorphic connection are not all contained in a single wall (or intersection of walls) of the cone of dominant \(P\)-characters.
\end{enumerate}
\end{lemma}
In particular, if \(M\) is compact and K\"ahler with trivial canonical bundle and a holomorphic parabolic geometry, then after perhaps replacing \(M\) by a finite \'etale covering space, \(M\) is a complex torus and \(E \to M\) is a translation invariant \((X,G)\)-geometry; see \cite{McKay:2008} p. 8 theorem 3 for a complete classification of these.
\begin{proof}
Let \(\trivChars\) be the kernel of the homomorphism
\[
\chi \in \Hom{P}{\C{\times}} \mapsto a\pr{L_{\chi}} \in \cohomology{1}{M,\Omega^1}.
\]
In other words, the group \(\trivChars\) consists of all characters \(\chi \in \Hom{P}{\C{\times}}\) whose associated line bundle \(L_{\chi} \to M\) admits a holomorphic connection.
Pick a dominant \(P\)-character in \(\trivChars\), say \(\chi=\sum p_i \omega_i\) in a basis of fundamental weights, with as many positive \(p_i\) as possible.
There is such a \(\chi\), because the sum of any two \(P\)-characters is a \(P\)-character.
Recall that \(P^{\chi} = M A \pr{N \cap M_{\chi}}\).
By proposition~\vref{proposition:canonical.connection}, there is a holomorphic \(P^{\chi}\)-reduction \(E^{\chi} \subset E\).
Conversely, every such reduction gives a holomorphic connection on \(L_{\chi}\).
If \(X_{\chi}=X\), then \(P^{\chi}=MA\) is a reductive complex Lie group, so \(E^{\chi} \to M\) is a Sharpe geometry (see section~\vref{section:Sharpe.geometry}) so \(E^{\chi} \to M\) admits a holomorphic connection, which induces a holomorphic connection on \(E \to M\) and on all \(L_{\chi}\) and on \(TM\).
If \(X_{\chi} \ne X\), then all \(\chi \in \Sigma\) are perpendicular to some simple root \(\alpha_i\), so all of the dominant ones lie in some wall.
\end{proof}

We prove theorem~\vref{theorem:Parabolic}.
\begin{proof}
Suppose that \(E \to M\) is a holomorphic \(\pr{X,G}\)-geometry on a domain \(M\) over a Stein manifold.
By lemma~\vref{proposition:canonical.connection}, \(E \to M\) admits a reductive reduction just when the canonical bundle of \(M\) admits a holomorphic connection.
By lemma~\vref{lemma:AtiyahClassZero}, every holomorphic vector bundle on a domain over a Stein manifold which is pulled back from the Stein manifold admits a holomorphic connection.
The canonical bundle of \(M\) is the pull back of the canonical bundle of the Stein manifold.
By lemma~\vref{lemma:extend.by.reducing}, every \pr{X,G}-geometry extends from \(M\) to its envelope of holomorphy.
\end{proof}

\begin{longversion}

\section{Families of first order structures and geometries transverse to a family of foliations}

A \emph{family of foliated complex manifolds} is a pair of nowhere singular holomorphic foliations \(\fol^-, \fol^+\) of a complex manifold \(M\), so that each leaf of \(\fol^-\) lies in a single leaf of \(\fol^+\).
We identify any foliation \(\fol\) with the set of tangent vectors tangent to its leaves, so \(\fol^- \subset \fol^+ \subset TM\).
The \emph{frame bundle} of the pair \(\fol=\pr{\fol^-,\fol^+}\) is the principal right \(\GL{n,\C{}}\)-bundle \(F=\Fr{\fol} \to M\) whose elements are pairs \((m,u)\) where \(m \in M\) and \(u \colon \fol^+_m/\fol^-_m \to \C{n}\) a linear isomorphism.
Suppose that \(\rho \colon G \to \GL{n,\C{}}\) is a complex analytic representation of a complex
Lie group \(G\).
A \emph{holomorphic family of transverse \(G\)-structures} is a family \(\fol\) of foliated complex manifolds and a holomorphic principal right \(G\)-bundle \(E \to M\) and a \(G\)-equivariant holomorphic bundle map \(E \to \Fr{\fol}\).

If \(p \colon E \to M\) is a holomorphic principal bundle and \(\fol=\pr{\fol^-,\fol^+}\) is a family of foliated manifolds on \(M\), let \(\fol^-_E, \fol^+_E\) be the foliations on \(E\) whose leaves are the preimages of the leaves of \(\fol^-, \fol^+\), and write \(\fol^-, \fol^+\) as \(\fol^-_M, \fol^+_M\) for clarity.
Suppose that \(E \to \Fr{\fol_M}\) is a family of \(G\)-structures.
Pick any \(e \in E\) and suppose that the bundle map \(E \to \Fr{\fol_M}\) takes \(e \in E \mapsto (m,u)\in \Fr{\fol_M}\).
The \emph{soldering form} of the family of \(G\)-structures is the holomophic section \(\eta\) of \(\pr{\fol^+_E/\fol^-_E}^* \otimes \C{n}\) defined by \(v \hook \eta = u\of{p'(e)v}\) for any \(v \in \fol^+_{E,e}\).

Suppose that \(X=G/H\) is a complex homogeneous space.
A \emph{family of \pr{X,G}-geometries} (also called a \emph{family of Cartan geometries} modelled on \((X,G)\)) on a family \(\fol_M=\pr{\fol^-_M,\fol^+_M}\) of foliated complex manifolds is a holomorphic principal right \(H\)-bundle \(E \to M\) and a section \(\omega\) of \(\pr{\fol_E^+/\fol_E^-}^* \otimes \mathfrak{g}\), called the \emph{Cartan connection}, satisfying the following conditions.
\begin{enumerate}
\item
Denote the right action of \(g \in G\) on \(e \in E\) by \(r_g e=eg\). 
The Cartan connection transforms in the adjoint representation:
\(
r_g^* \omega = \Ad_g^{-1} \omega.
\)
\item
On each leaf \(\Lambda^+\) of \(\fol^+_E\), treating \(\omega\) as a 1-form on \(\Lambda^+\), \(0= v \hook d \omega\) for all \(v \in \fol^-_E\).
\item
\(\omega_e \in \pr{\fol^+_{E,e}/\fol^-_{E,e}}^* \otimes \mathfrak{g}\) is a linear isomorphism at each point
\(e \in E\).
\item
For each \(A \in \mathfrak{g}\), define a section \(\dot{A}\) of \(\fol^+_E/\fol^-_E\) over \(E\) by the equation \(\dot{A} \hook \omega = A\). 
Define vector fields \(\vec{A}\) on \(E\), for \(A \in \mathfrak{h}\), generating the right \(H\)-action:
\[
\vec{A} = \left. \frac{d}{dt} r_{e^{tA}} \right|_{t=0}.
\]
Then \(\vec{A} + \fol^-_E = \dot{A}\).
\end{enumerate}
If \(G\) is a complex semisimple Lie group and \(X=G/P\) is a rational homogeneous variety, then a family of \pr{X,G}-geometries is called a \emph{family of parabolic geometries}.

Recall the foliation of \(\C{n}\wo\br{0}\) by radial lines.
Clearly foliations, and even holomorphic submersions, do not generally arise as pullbacks from envelopes of holomorphy. 
Therefore we will pose the Hartogs extension problem only for holomorphic foliations which are already assumed to be pulled back.

\begin{theorem}\label{theorem:ExtendByFirstOrderStructureFamily}
Suppose that \(G \subset \GL{n,\C{}}\) is a closed complex Lie subgroup.
Suppose that \(M\) is a domain over a Stein manifold and \(\fol^-_M \subset \fol^+_M\) are nowhere singular holomorphic foliations of \(M\).
Suppose that \(E \subset \Fr{\fol_M}\) is a holomorphic family of \(G\)-structures on \(M\).
Then the family is pulled back from a unique holomorphic family of \(G\)-structures on the envelope of holomorphy \(\hat{M}\) if and only if the holomorphic principal bundle \(E \to M\) is the pullback of a holomorphic principal bundle on \(\hat{M}\).
\end{theorem}
\begin{proof}
We just make the obvious modifications to the proof we gave above for a single \(G\)-structure.
Suppose that \(E \to M\) is the pullback of \(\hat{E} \to \hat{M}\).
(It is true that \(\hat{E}\) is the envelope of holomorphy of \(E\), although we won't need to know this, so we leave the reader to check it.)
The soldering form is a holomorphic section of the holomorphic vector bundle \(\pr{\fol^+_E/\fol^-_E}^* \otimes^G \C{n}\) over \(M\), and therefore is the pullback of a holomorphic section of \(\pr{\fol^+_{\hat{E}}/\fol^-_{\hat{E}}}^* \otimes^G \C{n}\) on \(\hat{M}\) by proposition~\vref{proposition:vector.bundle.sections}.
This section vanishes on the fibers of \(E \to M\), so by analytic continuation vanishes on all of the fibers of \(\hat{E} \to \hat{M}\). 
The set of points \(e \in \hat{E}\) at which this section is not of full rank as a linear map is a complex analytic hypersurface of \(\hat{E}\).
Being \(G\)-invariant, this hypersurface projects to a complex analytic hypersurface of \(\hat{M}\).
Every component of this hypersurface has to intersect \(M\) by lemma~\vref{lemma:hypersurfaceIntersections}.
Therefore this hypersurface is empty. 
The soldering form, being semibasic, determines a coframe at each point \(e \in \hat{E}\), say \(u \in \pr{\fol^+_{M,m}/\fol^-_{M,m}}^* \otimes \C{n}\) so that \(v \hook \eta_e = u\of{\pi'(e)v}\) for any vector \(v \in \fol^+_{\hat{E},e}/\fol^-_{\hat{E},e}\).
This choice of coframe is an extension of our map to the coframe bundle. 
The map is \(G\)-equivariant, so is a family of \(G\)-structures.
\end{proof}

\begin{theorem}\label{theorem:extendByBundle}
Suppose that \(H \subset G\) is a closed complex Lie subgroup of a complex Lie group and \(X=G/H\).
Suppose that \(M\) is a domain over a Stein manifold and \(\fol^-_{\hat{M}} \subset \fol^+_{\hat{M}}\) are nowhere singular holomorphic foliations of the envelope of holomorphy \(\hat{M}\) of \(M\).
Suppose that \(E \to M\) is a holomorphic family of \pr{X,G}-geometries over the pullback foliations \(\fol^-_M \subset \fol^+_M\).
Then the family of geometries is pulled back from the envelope of holomorphy \(\hat{M}\) if and only if the holomorphic principal bundle \(E \to M\) is pulled back from a holomorphic principal bundle on \(\hat{M}\).
\end{theorem}
\begin{proof}
Suppose that \(E \to M\) is pulled back from a holomorphic principal bundle \(\hat{E} \to \hat{M}\). 
The Cartan connection is a holomorphic section of \(\pr{\fol^+_E/\fol^-_E}^* \otimes^H \LieG \to M\), and therefore is pulled back from a holomorphic section \(\hat{\omega}\) of \(\pr{\fol^+_{\hat{E}}/\fol^-_{\hat{E}}}^* \otimes^H \LieG \to \hat{M}\) by proposition~\vref{proposition:vector.bundle.sections}.
The set of points \(e \in E\) at which the linear map
\[
v \in T_e \hat{E} \mapsto v \hook \hat{\omega} + \LieH \in \LieG/\LieH
\] 
is not a linear isomorphism is a closed complex analytic hypersurface of \(\hat{M}\).
Every component of this hypersurface has to intersect \(M\) by lemma~\vref{lemma:hypersurfaceIntersections}.
Therefore this hypersurface is empty. 
The object \(\omega\) satisfies \(\vec{A} \hook \omega=A\) for any \(A \in \LieH\) where \(\vec{A}\) is the associated generator of the infinitesimal \(\LieH\)-action on \(E\). 
The analogous statement on \(\hat{E}\) holds by analytic continuation.
Therefore \(\hat{\omega}\) is a Cartan connection.
\end{proof}

\begin{theorem}\label{theorem:ExtendFaithfulFamily}
Suppose that \(G\) is a complex Lie group and \(H \subset G\) is a is a closed complex subgroup, with Lie algebras \(\LieH \subset \LieG\) and let \(X\defeq G/H\).
Suppose that \(H\) acts faithfully on \(\LieG/\LieH\).
Suppose that \(M\) is a domain over a Stein manifold and \(\fol^-_{\hat{M}} \subset \fol^+_{\hat{M}}\) are nowhere singular holomorphic foliations on the envelope of holomorphy \(\hat{M}\) of \(M\).
Suppose that \(E \to M\) is a holomorphic family of \pr{X,G}-geometries over the pullback foliations \(\fol^-_M \subset \fol^+_M\).
Then the family of geometries is a pullback from the envelope of holomorphy \(\hat{M}\) if and only if the associated first order structure is the pullback of a holomorphic first order structure on \(\hat{M}\).
\end{theorem}
\begin{proof}
The family of Cartan geometries on \(M\) is pulled back from\(\hat{M}\) just when the associated principal \(H\)-bundle \(E \to M\) is pulled back by theorem~\vref{theorem:extendByBundle}.
The bundle of the associated first order structure is \(E/H_1=E\).
\end{proof}

\begin{theorem}
Suppose that \(M\) is a domain over a Stein manifold and \(\fol^-_{\hat{M}} \subset \fol^+_{\hat{M}}\) are nowhere singular holomorphic foliations of the envelope of holomorphy \(\hat{M}\) of \(M\).
Then every holomorphic family of effective reductive geometries with connected model over the pullback foliations is the pullback of a unique holomorphic family of reductive geometries over \(\hat{M}\).
\end{theorem}
\begin{proof}
By lemma~\vref{lemma:Faith}, \(H\) acts faithfully on \(\LieG/\LieH\). 
By theorem~\vref{theorem:ExtendFaithfulFamily}, the result follows.
\end{proof}

\begin{theorem}
Every holomorphic family of effective parabolic geometries on a domain over a Stein manifold is pulled back from a unique family on the envelope of holomorphy of the domain.
\end{theorem}
We omit the proof, which is once again a straightforward modification of the proof for a single effective parabolic geometry, using a holomorphic connection on the relative canonical bundle.

\end{longversion}

\section{Flat Cartan connections and disk convexity}

Sergei Ivashkovich has another perspective on holomorphic extensions which is easy to apply to developing maps of a wider variety of models.
A complex manifold \(M\) is \emph{disk convex} if for every compact set \(K \subset M\) there is a compact set \(K' \subset M\) so that every holomorphic map of the closed disk to \(M\) sending the boundary to \(K\) sends the interior to \(K'\); see Ivashkovich \cite{Ivashkovich:2008}.
For example, compact complex manifolds are disk convex, as are affine complex analytic varieties.
Products of disk convex manifolds are disk convex.
All Riemann surfaces and many complex homogeneous surfaces are disk convex \cite{McKay:2015L}.

\begin{theorem}%
[Ivashkovich \cite{Ivashkovich:2008}]%
\label{theorem:Ivashkovich}
Suppose that \(M\) is a domain over a Stein manifold and \(X\) is a disk convex K{\"a}hler manifold.
Suppose that the holomorphic vector fields on \(X\) span the tangent bundle of \(X\).
Every local biholomorphism \(M \to X\) extends uniquely to a local biholomorphism \(\hat{M} \to X\).
\end{theorem}

\begin{corollary}\label{corollary:extend.structure}
Suppose that \(M\) is a domain over a Stein manifold.
Suppose that \pr{X,G} is an effective complex homogeneous space and that \(X\) is disk convex and K{\"a}hler.
(For example, \pr{X,G} could be a product of a reductive complex homogeneous space with a rational homogeneous variety.)
Then every flat holomorphic \pr{X,G}-geometry on \(M\) is the pullback of a unique flat holomorphic \pr{X,G}-geometry on \(\hat{M}\).
\end{corollary}

\section{Example: Hopf manifolds}

\subsection{Definition}

A \emph{Hopf manifold} is a compact complex manifold \(M\) whose universal covering space is biholomorphic to \(\C{n}\wo\br{0}\) for some integer \(n \ge 1\).
It is \emph{primary} if its fundamental group is generated by a single element.

A biholomorphism \(f \colon \C{n} \to \C{n}\) fixing \(0 \in \C{n}\) is \emph{strictly contracting} if all of the eigenvalues \(\lambda\) of \(f'(0)\) have \(\left|\lambda\right|<1\). 
If \(f\) is strictly contracting, let 
\[
M_f \defeq
\pr{\C{n} \setminus \left\{0\right\}}/\pr{z \sim f(z)};
\]
then \(M_f\) is a primary Hopf manifold.
Every primary Hopf manifold is biholomorphic to some such \(M_f\) \cite{Kodaira:1966} p. 694, \cite{Hasegawa:1993} theorem 2.1.
Any two primary Hopf manifolds \(M_f\) and \(M_g\) are biholomorphic if and only if there is a biholomorphism \(h \colon \C{n} \to \C{n}\) fixing \(0\) so that
\(g \circ h = h \circ f\): the classification of primary Hopf manifolds is the classification of strictly contracting biholomorphisms of \(\C{n}\) up to conjugacy. 
Every Hopf manifold has a finite normal covering by a primary Hopf manifold \cite{Kodaira:1966} p. 694, \cite{Hasegawa:1993} theorem 2.1.

A \emph{resonance} of a biholomorphism \(f \colon \C{n} \to \C{n}\) fixing \(0 \in \C{n}\) is a relation of the form \(\lambda_j = \lambda^{\alpha}\) where \(\lambda_j\) are the eigenvalues of \(f'(0)\) and \(\alpha=\pr{\alpha_1,\alpha_2,\dots,\alpha_n}\) are some integers with \(\sum_k \alpha_k \ge 2\).
To each resonance we associate the \emph{resonant polynomial map} \(z\in\C{n} \mapsto z^{\alpha} e_j\in\C{n}\), where \(e_1, e_2, \dots, e_n \in \C{n}\) is the standard basis.

By the Poincar\'e--Dulac theorem \cite{Arnold:1988} p. 184, any strictly contracting biholomorphism \(f\) is conjugate (by a biholomorphism fixing the origin) to a strictly contracting biholomorphism of the form
\(z\in\C{n} \mapsto Az+w(z)\in\C{n}\), where \(A=f'(0)\) and \(w\) is a linear combination of resonant polynomial maps.
Order the eigenvalues of \(A\) by modulus.
Henceforth replace \(f(z)\) by \(Az+w(z)\), which doesn't change the biholomorphism type of the Hopf manifold.
The generic strictly contracting biholomorphism is linearizable, and conjugate to a diagonal linear map
\[
z \mapsto \pr{\lambda_1 z_1, \lambda_2 z_2, \dots, \lambda_n z_n}.
\]

A Hopf manifold \(M\) which is not primary is \emph{secondary}.
For each secondary Hopf manifold \(M\), there is a primary Hopf manifold \(M_f \to M\), a finite normal covering, so that \(f\) lies in the center of \(\fundamentalgroup{M}\) \cite{Kodaira:1966} p. 694.
If \(\Gamma\) is a group of biholomorphisms of \(\C{n}\) fixing the origin, properly discontinuously on \(\C{n}\wo{0}\), and containing a strictly contracting map, then \(M_{\Gamma}=\pr{\C{n}\wo{0}}/\Gamma\) is a Hopf manifold, and every Hopf manifold arises as \(M_{\Gamma}\).
The group \(\Gamma\) is defined up to conjugacy by a biholomorphism of \(\C{n}\) fixing \(0\).

The basin of attraction \(X_o\) of an attractive fixed point \(o \in X\) of a biholomorphism \(f \colon X \to X\) of a complex manifold \(X\) is biholomorphic to complex Euclidean space. (This result was proved in \cite{Sternberg:1957} and independently in \cite{Rosay/Rudin:1988} but not clearly stated; for a clear statement see \cite{Abate/Abbondandolo/Majer:2014}.)
Therefore the quotient 
\[
M_f \defeq \pr{X_o-o}/\pr{x \sim f(x)}
\]
is a primary Hopf manifold.
Take a discrete group \(\Gamma\) of biholomorphisms of a complex manifold \(X\) fixing a point \(o \in X\), containing an element which is strictly contracting toward \(o\), and acting properly discontinuously away from \(o\) on the basin of attraction \(X_o\) of the strictly contracting element.
The quotient 
\[
M_{\Gamma} \defeq \pr{X_o-o}/\Gamma
\]
is a Hopf manifold.

\subsection{Holomorphic affine connections}

Take a group \(\Gamma \subset \GL{n,\C{}}\) containing a central strictly contracting linear map, and acting properly discontinuously on \(\C{n}\wo{0}\).
Call \(M_{\Gamma}\) a \emph{linear Hopf manifold}.
Clearly \(M_{\Gamma}\) bears  holomorphic affine connection, which pulls back to \(\C{n}\wo{0}\) to the standard flat affine connection on \(\C{n}\).

\begin{theorem}
Suppose that \(M\) is a real manifold with an affine connection \(\nabla\) and pick a point \(m_0 \in M\).
Let \(G\) be the set of diffeomorphisms \(f\) defined near \(m_0\), fixing \(m_0\) and preserving \(\nabla\).
Similarly let \(\LieG\) be the set of vector fields defined near \(m_0\), vanishing at \(m_0\) with flow preserving \(\nabla\).
In any system of geodesic normal coordinates near \(m_0\), every element of \(G\) is an invertible linear transformation and every element of \(\LieG\) is a  linear vector field.
\end{theorem}
\begin{proof}
This result is stated without proof in \cite{Szaro:1998} lemma 3.1 and \cite{Dumitrescu/Guillot:2013} lemma 7.
Geodesic normal coordinates are precisely those in which the geodesics through the origin are straight lines, parameterized linearly.
This is clear geometrically, because in any such coordinates the exponential map is the identity map, and conversely.
(One can also find a detailed proof in  \cite{Wolf:1967} p. 23 theorem 1.6.20.)
Any \(f \in G\) takes each geodesic \(t \mapsto tv\) in those coordinates to some geodesic in the same coordinates, \(t \mapsto tw=f(tv)\), for all small enough \(t\).
Differentiate and set \(t=0\): \(w=f'(0)v\).
So \(f(tv)=tw=tf'(0)v\) for all small enough \(t\), i.e. \(f(v)=f'(0)v\) for \(v\) near \(0\), i.e. \(f\) is linear in those coordinates.
For any vector field \(X \in \LieG\), the argument above can be applied to the flow of \(X\), defined near \(m_0\) for small enough times.
The flow is linear, and we differentiate the flow as \(t \to 0\) to see that \(X\) is linear.
\end{proof}

For example, if the affine connection is holomorphic on a complex manifold, the biholomorphisms preserving a point and the connection are simultaneously complex linear in the holomorphic geodesic normal coordinates.
Similarly any compact group of diffeomorphisms preserving a point must preserve a Riemannian metric, so an affine connection, so lie in the orthogonal group in suitable geodesic normal coordinates, recovering Bochner's theorem \cite{Bochner:1945}.

\begin{lemma}\label{lemma:Hopf.connection.iff.linear}
Any Hopf manifold \(M\) admits a holomorphic affine connection if and only if it is biholomorphic to a linear Hopf manifold.
\end{lemma}
\begin{proof}
If \(\dimC{M}=1\) then \(M\) is an elliptic curve, so a linear Hopf manifold, and admits a holomorphic affine connection; so assume that \(\dimC{M} \ge 2\).
Any holomorphic connection \(\nabla\) on the tangent bundle of a Hopf manifold \(M=M_{\Gamma}\) lifts to \(\C{n}\wo{0}\).
The connection then differs from the standard Euclidean connection on \(\C{n}\) by a holomorphic 1-form valued in \(TM\).
Apply Hartogs's extension theorem to this 1-form: the connection holomorphically extends to a \(\Gamma\)-invariant holomorphic connection on \(\C{n}\).
Therefore the exponential map identifies all \(f \in \Gamma\) near \(0\) with linear maps.
\end{proof}

To each holomorphic affine connection, say with curvature \(K^i_{jk\ell}\), we associate the unsymmetrized Ricci tensor with components \(K^k_{jk\ell}\), which splits into symmetric and antisymmetric tensors.
On any linear Hopf manifold, this tensor vanishes because it lifts to a holomorphic tensor on \(\C{n}\) invariant under a strictly contracting map; the affine connections on a Hopf manifold are Ricci flat.
Similarly, the torsion, with components \(t^i_{jk}\), has a trace \(t^k_{jk}\), which vanishes.

\begin{lemma}
Suppose that \(f\) is a diagonalizable linear map with eigenvalues \(\lambda_1, \lambda_2, \dots, \lambda_n\), \(n \ge 2\).
The moduli space of holomorphic affine connections on \(M_f\) has dimension equal to the number of relations of the form \(\lambda_i = \lambda_j \lambda_k \lambda^{\alpha}\),
for \(1 \le i, j, k \le n\) and a multiindex \(\alpha\).
(This number is always finite.)
If a diagonalizable strictly contracting map \(f\) admits no such relations, then the primary Hopf manifold \(M=M_f\) admits a unique holomorphic affine connection; this connection lifts to \(\C{n}\) to be the standard flat affine connection \(\nabla_{\partial_{z_j}} = \partial_{z_j}\).
\end{lemma}
\begin{proof} 
The \(f\)-invariance of a connection is precisely that the Christoffel symbols, expanded into a multiindex Taylor series
\[
\Gamma^i_{jk}(z) = \sum_{\alpha} \Gamma^i_{jk\alpha} \frac{z^{\alpha}}{\alpha!},
\]
satisfy \(\Gamma^i_{jk \alpha} = 0\) unless \(\lambda_i = \lambda_j \lambda_k \lambda^{\alpha}\).
\end{proof}

\begin{lemma}\label{lemma:one.connection.on.K}
Every line bundle on a Hopf manifold of complex dimension 2 or more admits a unique holomorphic connection, and this connection is flat.
\end{lemma}
\begin{proof}
Every holomorphic line bundle on any primary Hopf manifold admits a holomorphic connection; see \cite{Mall:1991} p. 1013 theorem 4.
This holomorphic connection is unique, because the difference between any two holomorphic flat connections is a holomorphic 1-form, i.e. a holomorphic 1-form on \(\C{n}\) invariant under a strictly contracting map, so vanishes.
The curvature of a holomorphic connection on a line bundle on a Hopf manifold is a holomorphic 2-form, and lifts to a holomorphic 2-form on \(\C{n}\wo{0}\) invariant under a strictly contracting map, so vanishes.

Consider a secondary Hopf manifold.
Pullback the line bundle back to a primary Hopf covering space, where it has a unique holomorphic connection.
By uniqueness, this connection is invariant under the deck transformations.
\end{proof}

\begin{proposition}
Every holomorphic projective connection on any Hopf manifold of complex dimension 2 or more is induced by a unique holomorphic affine connection.
Consequently, any Hopf manifold of complex dimension 2 or more bears a holomorphic projective connection just when it is linear, and the moduli space of holomorphic projective connections is then identified with the moduli space of holomorphic affine connections.
\end{proposition}
\begin{proof}
Start with a holomorphic projective connection, say with Cartan connection \(\omega\) and principal bundle \(E \to M\).
The structure equations in the notation of Cartan \cite{Cartan:1992} chapter IV, p. 234 are
\begin{align*}
d \omega^i &= - \omega^i_j \wedge \omega^j + \frac{1}{2}t^i_{jk} \omega^j \wedge \omega^k, \\
d \omega^i_j &= - \omega^i_k \wedge \omega^k_j + 
\pr{\delta^i_j \omega_k + \delta^i_k \omega_j}\wedge \omega^k
+
\frac{1}{2}t^i_{jk\ell} \omega^k \wedge \omega^{\ell}, \\
d \omega_i &= \omega^j_i \wedge \omega_i + \frac{1}{2}t_{ijk} \omega^j \wedge \omega^k.
\end{align*}
Take the unique holomorphic connection \(\kc\) on the canonical bundle guaranteed by lemma~\ref{lemma:one.connection.on.K} and reduce the structure group holomorphically to the maximal reductive subgroup as in proposition~\vref{proposition:canonical.connection}. 
The subbundle \(E' \subset E\) is then \(E'=FM\), the coframe bundle, because the maximal reductive subgroup is \(\GL{n,\C{}}\).
The reduction \(E' \subset E\) is expressed as \(\omega_i = t_{ij} \omega^j\), for some uniquely determined holomorphic functions \(t_{ij} \colon E_0 \to \C{}\).
From the structure equations \(t_{ij} \omega^i \omega^j\) is the pullback from \(M\) of a unique holomorphic 2-tensor in \(\Sym{2}{T^*M}\).
Such a tensor pulls back to \(\C{n}\), the universal covering space of \(M\), to become a \(\Gamma\)-invariant holomorphic 2-tensor, so vanishes.
Therefore on \(E'\), \(\omega_i=0\).
The structure equations then immediately become those of a holomorphic affine connection:
\begin{align*}
d \omega^i &= - \omega^i_j \wedge \omega^j + \frac{1}{2}t^i_{jk} \omega^j \wedge \omega^k, \\
d \omega^i_j &= - \omega^i_k \wedge \omega^k_j + \frac{1}{2}t^i_{jk\ell} \omega^k \wedge \omega^{\ell}.
\end{align*}
Pull back the bundle to \(\C{n}\) and take complex linear coordinates \(z\) on \(\C{n}\).
Pull back \(\omega^i\) and \(\omega^i_j\) via the section \(dz\) of \(FM\).
Then \(dz^* \omega^i = dz^i\) and so by Cartan's lemma \(dz^* \omega^i_j = \Gamma^i_{jk} \omega^k\) for some functions \(\Gamma^i_{jk}(z)\) with \(dz^* t^i_{jk} = \Gamma^i_{jk} - \Gamma^i_{kj}\).
Clearly the projective connection is induced by this affine connection.
Two affine connections induce the same projective connection just when they differ by a 1-form, and there are no nonzero holomorphic 1-forms on a Hopf manifold, so the affine connection is unique.
\end{proof}
For the classification of holomorphic Cartan geometries on Hopf manifolds of complex dimension one, i.e. elliptic curves, see \cite{McKay2011}.

\subsection{Reductive first order structures}

\begin{example}\label{example:HopfReductiveStructure}
Suppose that \(G \subset \GL{n,\C{}}\) is a closed complex subgroup and that \(\Gamma \subset G\) is a subgroup containing a strictly contracting linear map and acting properly discontinuously on \(\C{n}\wo{0}\).
Consider the quotient 
\[
B = \prodquot{
{\pr{\C{n}\wo{0}}}}
{G}
{\Gamma}
\]
where the \(\Gamma\)-action is generated by
\[
\pr{z,u} \mapsto \pr{gz,ug^{-1}}.
\]
Clearly \(B \subset FM_{\Gamma}\) is a \(G\)-structure on \(M_{\Gamma}\).
\end{example}

\begin{theorem}\label{theorem:reductive.1st.Hopf}
Suppose that \(G \subset \GL{n,\C{}}\) is a reductive subgroup, \(n \ge 2\).
Suppose that \(B \subset FM\) is a \(G\)-structure on a Hopf manifold \(M=M_{\Gamma}\).
Then up to isomorphism of \(G\)-structures, \(f'(0) \in G\) for all \(f \in \Gamma\). 
If \(M\) is a linear Hopf manifold then, up to biholomorphism, \(\Gamma \subset G\).
If \(G\) is semisimple, or has finitely many components and semisimple identity component, then there are no holomorphic \(G\)-structures on any Hopf manifold.
\end{theorem}
\begin{proof}
This structure \(B\) pulls back to a \(G\)-structure on \(\C{n}\wo{0}\).
Identify 
\[
F\of{\C{n}\wo{0}}
=
\pr{\C{n}\wo{0}} \times \GL{n,\C{}}.
\]
The prolongation of any \(f \in \Gamma\) is
\[
f_1\of{z,u}=\pr{f(z),u f'(z)^{-1}}.
\]
For each \((z,u) \in B\),
\[
\pr{f(z),u f'(z)^{-1}} \in B.
\]
Since \(G\) is reductive, the pullback of \(B\) to \(\C{n}\wo\br{0}\) extends across \(0\), so we can always linearly change the coordinates in \(\C{n}\) to arrange that \((0,I) \in B\), and \(g_0=f'(0) \in G\).
For some \(f \in \Gamma\), \(g_0\) is a strictly contracting linear map, i.e. all of the eigenvalues of \(g_0\) lie inside the open unit disk in the complex plane.
If \(G\) has semisimple identity component, and finitely many components, then all representations of \(G\) are virtually unimodular, so no elements are strictly contracting on \(\C{n}\).
\end{proof}

\subsection{Reductive geometries}

\begin{example}\label{example:reductive.geometry.on.Hopf}
Suppose that \(G\) is a complex Lie group, \(H \subset G\) a reductive complex algebraic subgroup, and let \(X\defeq G/H\).
Suppose that \(G\) acts faithfully on \(X\).
Let \(\rho \colon H \to \GL{\LieG/\LieH}\) be the obvious representation.
Take a subgroup \(\Gamma \in H\) acting properly discontinuously on \(\LieG/\LieH\wo{0}\) and containing an element which acts on \(\LieG/\LieH\wo{0}\) as a strictly contracting linear map.
Let \(M\defeq M_{\rho\pr{\Gamma}}\) and  \(E \defeq \prodquot{\pr{\C{n}\wo{0}}}{H}{\Gamma}\), the quotient by the left action \(h_0\pr{z,h}=\pr{h_0z,h_0h}\).
Let \(\omega \in \nForms{1}{\pr{\C{n}\wo{0}} \times H} \otimes \LieG\) be
\(\omega_{(z,h)} \defeq h^{-1} \, dh\).
Clearly \(\omega\) is \(\Gamma\)-invariant, and therefore descends to a 1-form on \(E\), which we also call \(\omega\), and which is a Cartan connection for a unique flat \pr{X,G}-geometry on \(M\).
\end{example}

\begin{theorem}\label{theorem:reductive.geometry.Hopf}
Suppose that \(G\) is a complex Lie group, \(H \subset G\) a reductive complex algebraic subgroup, and let \(X\defeq G/H\).
Suppose that \(G\) acts faithfully on \(X\).
Let \(\rho \colon H \to \GL{\LieG/\LieH}\) be the obvious representation.
A Hopf manifold \(M\) admits a holomorphic \pr{X,G}-geometry if and only if it is biholomorphic to \(M=M_{\rho\pr{\Gamma}}\) for a subgroup \(\Gamma \subset H\), unique up to conjugacy in \(H\).
\end{theorem}
\begin{proof}
Existence follows from example~\ref{example:reductive.geometry.on.Hopf} above.
On the other hand, if \(M=M_{\Gamma}\) admits an \pr{X,G}-geometry, lemma~\vref{lemma:reductive.geometry.induces.affine.connection} shows that \(M_{\Gamma}\) admits a holomorphic affine connection, and lemma~\vref{lemma:Hopf.connection.iff.linear} says that \(M_{\Gamma}\) admits a holomorphic affine connection if and only if, modulo isomorphism, \(\Gamma \subset \GL{\LieG/\LieH}\).
Theorem~\vref{theorem:reductive.1st.Hopf} says \(\Gamma \subset \rho{H}\).
Lemma~\vref{lemma:Faith} says that \(\rho\) is faithful on \(H\), so \(\Gamma \subset H\).
\end{proof}

\subsection{Some new parabolic geometries}

\begin{example}\label{example:flat.on.Hopf}
Take a complex semisimple Lie group \(G\) in its adjoint form and a parabolic subgroup \(P\) with Langlands decomposition \(P=MAN\) and a subgroup \(N^- \subset G\) 
\begin{longversion}
as in lemma~\vref{lemma:Langlands}, 
\end{longversion}
with Lie algebras \(\LieP = \LieM \oplus \LieA \oplus \LieN\) and \(\LieG = \LieN^- \oplus \LieM \oplus \LieA \oplus \LieN\).
Pick any element \(m_0 a_0 \in MA\) so that the map \(f=\Ad\pr{m_0 a_0}\) acts on \(\LieN^-\) as a strictly contracting linear map.
Such elements exist, and form an open subset in \(MA\), as we will see shortly.
Pick any subgroup \(\Gamma \subset MA\) containing \(m_0 a_0\) and acting propery discontinuously on \(\LieN^-\).
Take the Hopf manifold \(M=M_{\rho\pr{\Gamma}}\).
Let \(X\defeq G/P\) and let \(o \defeq 1 \cdot P \in X\).
Let \(\delta(Z) \defeq e^Z o \colon \LieN^- \to X\).
Let \(h \defeq \ad\pr{m_0 n_0} \colon N^- \to N^-\).
Since \(N^-\) is simply connected, the exponential map \(\exp \colon \LieN^- \to N^-\) is a diffeomorphism (Knapp \cite{Knapp:2002} p. 107 theorem 1.127) and therefore a biholomorphism.
Moreover, \(N^- \cap P = \br{1}\), so that \(\delta\) is a biholomorphism to its image.
The image of \(\delta\) is the \(N^-\)-orbit of \(o \in X\), which is also the basin of attraction of \(\ad\pr{m_0 n_0}\), as we will see in theorem~\vref{theorem:basin}.
The pair \pr{\delta,h} is an \pr{X,G}-structure.
This construction provides the first examples of flat holomorphic parabolic geometries on compact complex manifolds (other than the obvious example of \(M=X\)) for every rational homogeneous variety \pr{X,G} which is not a cominiscule variety.
\end{example}

Let us ask if there are any elements \(m_0 a_0 \in MA\) which act on \(\LieN^-\) as strictly contracting linear maps.
If so, then such elements will form an open set.
The generic element of \(MA\) is conjugate to an element in the Cartan subgroup of \(G\), so we can assume it has the form
\[
m_0 = e^{\sum_i c_i H_{\alpha_i}}, \
a_0 = e^{\sum_i c_i H_{\alpha_i}},
\]
where the first sum is over compact positive simple roots, and the second over noncompact positive simple roots.
We can then check that the eigenvalues with which such an element acts on \(\LieN^-\) are given, on the \(-\beta\) weight space, by \(\lambda_{\beta} = \exp\KillingForm{-\beta}{\sigma}\)
where
\[
\sigma = \sum_i \frac{c_i}{\KillingForm{\alpha_i}{\alpha_i}}.
\]
We can pick the constants \(c_i \in \C{}\) as we like, as long as \(\KillingForm{\beta}{\sigma} > 0\) for every positive noncompact root \(\beta\).
For example, if we pick the \(c_i\) so that \(\sigma=\delta\), then by lemma~\vref{lemma:DeltaBeta}, we find that \(m_0 a_0\) acts on \(\LieN^-\) by a strictly contracting linear map.

\subsection{Parabolic geometries}

\begin{theorem}\label{theorem:which.Hopfs.are.parabolic}
Suppose that \(M\) is a Hopf manifold admitting a holomorphic parabolic geometry \(E \to M\), say with model \(X=G/P\).
Let \(\rho \colon P \to \GL{\LieG/\LieP}\) be the obvious representation.
Split \(P\) into the Langlands decomposition \(P=MAN\).
Then \(M=M_{\rho\pr{\Gamma}}\) for some subgroup \(\Gamma \subset MA\) as in example~\vref{example:flat.on.Hopf}. 
A Hopf manifold admits a holomorphic parabolic geometry if and only if it admits a flat one.
The generic primary Hopf manifold of complex dimension 2 or more admits precisely one holomorphic parabolic geometry: its holomorphic flat projective connection.
\end{theorem}
\begin{proof}
By lemma~\vref{lemma:one.connection.on.K}, the canonical bundle of \(M\) admits a unique holomorphic flat connection.
As in proposition~\vref{proposition:canonical.connection}, the holomorphic connection on the canonical line bundle determines a holomorphic reductive \(MA\)-geometry \(E_0 \subset E\).
By theorem~\vref{theorem:reductive.geometry.Hopf}, \(M=M_{\rho\of{\Gamma}}\) for some \(p_0=m_0a_0\).
Example~\vref{example:flat.on.Hopf} shows that such Hopf manifolds admit flat holomorphic parabolic geometries.

Suppose that \(\pr{X,G} \ne \pr{\Proj{n},\PSL{n+1,\C{}}}\).
We want to prove that the generic primary Hopf manifold does not admit an \pr{X,G}-geometry.
It suffices to prove that the generic linear map does not occur in the form \(\rho(ma)\) for some \(ma \in MA\).
The generic element of \(MA\) is conjugate to an element of the Cartan subgroup, having the form \(e^H\) for some \(H\) in the Cartan subalgebra.
But then \(\lb{H}{X_{\alpha}}=\alpha(H)X_{\alpha}\) for each root, so the eigenvalues of \(H\) acting on \(\LieG/\LieP=\LieN^-\) are constrained to be the values of roots on root vectors.
If some positive root is the sum of any two other positive roots, this will force a relation between eigenvalues of \(H\), and generically chosen eigenvalues won't fit this pattern.
So if conjugates of \(MA\) are Zariski dense in \(\GL{\LieN^-}\), then no positive root can be written as a sum of positive roots.
So \(X\) is a cominiscule variety.
In any cominiscule variety other than \(\Proj{n}\), the action of \(MA\) on \(\LieN^-\) preserves a quadratic form up to scaling or a tensor product structure, forcing relations among the eigenvalues. 
So the generic matrix in \(\GL{n,\C{}}\) appears in the \(MA\)-action on \(\LieN^-\) just when \(\pr{X,G}=\pr{\Proj{n},\PSL{n+1,\C{}}}\), i.e. the parabolic geometry is a projective connection.
\end{proof}

\begin{theorem}
Every flat holomorphic parabolic geometry on any Hopf manifold is one of those in example~\vref{example:flat.on.Hopf}.
\end{theorem}
\begin{proof}
Take a Hopf manifold \(M=M_{\Gamma}\) with a flat holomorphic parabolic geometry.
By theorem~\vref{theorem:which.Hopfs.are.parabolic}, we can assume that each element \(f \in \Gamma\) is \(f=\Ad\pr{m_0 a_0}\) for some \(m_0 a_0 \in MA\).
Pick one element \(f \in \Gamma\) strictly contracting on \(\LieG/\LieH\).
By lemma~\vref{lemma:one.connection.on.K}, there is precisely one holomorphic connection on the canonical bundle of a primary Hopf manifold.
Use it to reduce the parabolic geometry \(E \to M\) to the Sharpe geometry \(E_{MA} \subset E\).
On the Sharpe geometry \(\omega=\omega_-+\omega_0+\omega_+\) and \(\omega_+ = t \omega_-\).
It is easy to check that \(t\) is a holomorphic section of \(T^* \otimes T^*\).
Lift to \(\C{n}\wo{0}\), extend by Hartogs extension theorem, and expand into a Taylor series on \(\C{n}\).
By \(f\) invariance, all terms in this Taylor series vanish: \(t\equiv 0\).
Therefore the reduction \(E_{MA} \subset \delta^* G\) is mapped by
\[
\begin{tikzcd}
E_{MA} \arrow{r} & \delta^* G \arrow{d} \arrow{r} & G \arrow{d} & H \arrow{l} \\
                 & \C{n} \arrow{r}{\delta} & X
\end{tikzcd}
\]
as an integral manifold of the Pfaffian system \(\omega_+=0\).
This left invariant Pfaffian system satisfies the conditions of the Frobenius theorem, and its maximal connected integral manifolds are precisely the left translates of \(N^-MA\).
Therefore the image of \(E_{MA}\) lies inside \(N^-MA \subset G\).
Let \(X_o \subset X\) be the \(N^-\)-orbit of \(o \in X\).
Clearly \(X_o\) is also the \(N^-MA\)-orbit of \(o \in X\).
Since \(\delta(0)=o\) lies in \(X_o\), the image of \(\delta\) lies in \(X_o\).
The map \(\log \delta \colon \C{n} \to \C{n}\) is a biholomorphism identifying the Hopf manifold \(M_{\Gamma}\) with the Hopf manifold \(M_{\Ad\pr{\Gamma}}\), and identifying the flat parabolic geometries.
\end{proof}

The moduli space of holomorphic parabolic geometries (or of holomorphic reductive geometries, or of flat holomorphic reductive geometries) is a complicated affine variety defined by hyperresonance conditions which we leave the reader to work out in detail; above we have determined when the moduli space is not empty.
Even the dimension of the moduli space is given by elaborate hyperresonance conditions, which we also leave to the reader.

\subsection{Flat Cartan connections and basins of attraction}

We understand very little about the moduli space of nonparabolic geometries on Hopf manifolds, but for many nonparabolic geometries we can relate the moduli problem to a problem in dynamics.

\begin{theorem}\label{theorem:basin}
Suppose that \pr{X,G} is a complex homogeneous space and that \(X\) is disk convex and K{\"a}hler and that \(G\) acts effectively on \(X\).
Suppose that \(M_{\Gamma}\) is a Hopf manifold of complex dimension at least 2, bearing a flat holomorphic \pr{X,G}-geometry with developing map \(\delta \colon \C{n}-0 \to X\) and holonomy morphism \(h \colon \Gamma \to G\), i.e. \(\delta(f(z))=h \, \delta(z)\) for all \(z \in \C{n}-0\). 
Then \(\delta\) has a unique holomorphic extension to a holomorphic map \(\delta \colon \C{n} \to X\).
Let \(o \defeq \delta(0) \in X\) and let \(H\) be the stabilizer of \(o\), so \(h\pr{\Gamma} \subset H\).
Let \(X_o\) be the basin of attraction of \(h\) toward \(o\); we have a commutative diagram of biholomorphisms
\[
\begin{tikzcd}
\pr{\C{n},0} \arrow{d}{\delta} \arrow{r}{f} & \pr{\C{n},0} \arrow{d}{\delta} \\
\pr{X_o,o} \arrow{r}{h(f)} & \pr{X_o,o}
\end{tikzcd}
\]
for \(f \in \Gamma\).
The Hopf manifold \(M_{\Gamma}\) is biholomorphic to  the Hopf manifold \(M_{h\of{\Gamma}}\) and has isomorphic flat \pr{X,G}-geometry.
In particular, \(f'(0) \colon \C{n} \to \C{n}\) has the same Jordan normal form as 
\(\Ad(h(f)) \colon \LieG/\LieH \to \LieG/\LieH\) for every \(f \in \Gamma\).
\end{theorem}
\begin{proof}
Corollary~\vref{corollary:extend.structure} gives the required extension of \(\delta\).
If we apply \(h(f)\) enough times, for some strictly contracting \(f\), then every point in the basin of attraction of \(o\) is eventually mapped to any arbitrarily small neighborhood of \(o\), and we can pick such a small neighborhood so that \(\delta\) maps onto that neighborhood, since \(\delta(0)=o\) and \(\delta\) is a local biholomorphism. 
The image of \(\delta\) is invariant under \(h\) and so under \(h^{-1}\), so equals the basin of attraction.
Suppose that \(\delta\pr{z}=\delta\pr{z'}\).
Then \(\delta\pr{f\pr{z}}=\delta\pr{f\pr{z'}}\), etc. so that we can arrange \(z\) and \(z'\) to lie in some arbitrarily small neighborhood of \(o\), in which \(\delta\) is a biholomorphism, so \(z=z'\). 
Therefore \(\delta\) is a biholomorphism onto the basin of attraction of \(h\).
\end{proof}

For example, if \(M_{\Gamma}\) has a flat holomorphic \pr{X,G}-geometry and \pr{X,G} is an equivariant product \(X=X_1 \times X_2, G=G_1 \times G_2\), then for every \(f \in \Gamma\), \(f'(0)\) has at least two Jordan blocks.
Moreover, \(h(f)\) is conjugate to a strictly contracting map of complex Euclidean space on each factor.

If \(X=G/H\) and an element of \(H\) is not strictly contracting on \(\LieG/\LieH\), then that element cannot be a holonomy image of a strictly contracting map from any flat holomorphic \pr{X,G}-geometry on any Hopf manifold.
The classification of flat \pr{X,G}-geometries on primary Hopf manifolds, for disk convex K\"ahler \(X\), reduces to the classification of the elements of \(G\) with an attractive fixed point \(o \in X\). For example, if \pr{X,G} is a rational homogeneous variety, \(X=G/P\), with Langlands decomposition \(P=MAN\), then the only elements of \(G\) which act as strictly contracting maps with a fixed point are, up to conjugation, certain elements of \(MA\).

\section{Conclusion}

The open problem of classifying holomorphic parabolic geometries on smooth complex projective varieties motivated the above research.
In section~\vref{section:blow.up}, we saw that there are new constraints which demonstrate that many complex manifolds have no parabolic geometries. 
On any manifold with a holomorphic parabolic geometry, we can cover the manifold with open sets so that we can pick a holomorphic section for the canonical bundle on each open set.
This reduces the holomorphic parabolic geometry on each open set to its Sharpe geometry.
On the overlaps of the open sets, the differences of the holomorphic connections on the canonical bundle give a cochain valued in the holomorphic 1-forms; this cochain encodes the ``glue'' to bind those Sharpe geometries into a single holomorphic parabolic geometry.
This approach should help to clarify the constraints on holomorphic parabolic geometries on smooth projective varieties.

\bibliographystyle{amsplain}
\bibliography{hartogs-parabolic}

\end{document}